\documentclass[reqno]{amsart}
\usepackage{amssymb,mathtools}
\usepackage[british]{babel}
\usepackage{enumitem}
\usepackage[latin1]{inputenc}
\usepackage{url}

\newtheorem{theorem}{Theorem}
\newtheorem*{theorem*}{Theorem}
\numberwithin{theorem}{section}
\newtheorem{corollary}[theorem]{Corollary}
\newtheorem{lemma}[theorem]{Lemma}
\newtheorem{proposition}[theorem]{Proposition}
\theoremstyle{definition}
\newtheorem{definition}[theorem]{Definition}
\newtheorem{remark}[theorem]{Remark}

\newcommand{\ord}{\textsf{Ord}}
\newcommand{\nat}{\textsf{Nat}}
\newcommand{\en}{\operatorname{en}}
\newcommand{\cl}{\operatorname{Cl}}
\newcommand{\cls}{\operatorname{cl}}

\newcommand{\tr}{\operatorname{Tr}}
\newcommand{\supp}{\operatorname{supp}}
\newcommand{\rng}{\operatorname{rng}}

\title[Patterns of resemblance and Bachmann-Howard fixed points]{Patterns of resemblance and\\ Bachmann-Howard fixed points}
\author{Anton Freund}

\begin{document}

\begin{abstract}
Timothy Carlson's patterns of resemblance employ the notion of $\Sigma_1$-elementarity to describe large computable ordinals. It has been conjectured that a relativization of these patterns to dilators leads to an equivalence with $\Pi^1_1$-comprehension (Question~27 of A.~Montalb\'an's ``Open questions in reverse mathematics", Bull.~Symb.~Log.~17(3)2011, 431-454). In the present paper we prove this conjecture. The crucial direction of the equivalence (towards $\Pi^1_1$-comprehension) is reduced to a previous result of the author, which is concerned with relativizations of the Bachmann-Howard ordinal.
\end{abstract}


\maketitle

\section{Introduction}

In their simplest form, T.~Carlson's patterns of resemblance are defined as follows (cf.~\cite[Section~10]{carlson-resemblance}): Consider the language $\mathcal L=\{\leq,\leq_1\}$ with two binary relation symbols. We only interpret this language in structures that have a set of ordinal numbers as universe, and $\leq$ is always interpreted as the usual order between \mbox{ordinals}. Let us agree that each ordinal is identified with its set of predecessors. We now determine the interpretation of $\leq_1$ by the recursive clause
\begin{equation*}
\alpha\leq_1\beta\quad:\Leftrightarrow\quad\text{$\alpha$ is a $\Sigma_1$-elementary $\mathcal L$-substructure of~$\beta$}.
\end{equation*}
The right side is equivalent to the conjunction of $\alpha\leq\beta$ and the following: For all finite $X\subseteq\alpha$ and $Y\subseteq\beta\backslash\alpha$, there is a $\widetilde Y\subseteq\alpha$ and a bijection $f:X\cup Y\to X\cup\widetilde Y$ that fixes~$X$ and satisfies $f(\gamma)\leq f(\gamma')\Leftrightarrow\gamma\leq\gamma'$ and $f(\gamma)\leq_1f(\gamma')\Leftrightarrow\gamma\leq_1\gamma'$ for all $\gamma,\gamma'\in X\cup Y$. Note that the given condition does only depend on the restriction of~$\leq_1$ to~$\beta\times\beta$, which can thus be defined by recursion on~$\beta$.

Patterns of resemblance are attractive due to their connections with several different areas: Carlson first used them to show that epistemic arithmetic is consistent with the statement ``I know that I am a Turing machine" (known as Reinhardt's strong mechanistic thesis, see~\cite{carlson-reinhardt}). They also offer a new approach to the large computable ordinals considered in ordinal analysis, including a conjectured characterization of the proof-theoretic ordinal of full second order arithmetics (see~\cite[Section~13]{carlson-resemblance}). This viewpoint has been clarified and expanded in subsequent work of G.~Wilken (see~\cite{wilken-bachmann-howard,wilken-assignments,wilken-skolem-hull,wilken-order2} and the joint paper~\cite{wilken-carlson-normal-form} with Carlson). Furthermore, Carlson has pointed out similarities with the core models studied in set theory, and expressed the hope that proof-theoretic and set-theoretic approaches ``will find common ground someday" (see again~\cite[Section~13]{carlson-resemblance}). The present paper establishes connections with weak set theories and reverse mathematics (a research program developed by H.~Friedman~\cite{friedman-icm-rm} and S.~Simpson, cf.~his textbook~\cite{simpson09}).

For the version of patterns that we have presented above, the smallest ordinal~$\alpha$ such that $\alpha\leq_1\beta$ holds for all $\beta\geq\alpha$ is equal to $\varepsilon_0=\min\{\gamma\,|\,\omega^\gamma=\gamma\}$, the proof-theoretic ordinal of Peano arithmetic (see~\cite[Section~10]{carlson-resemblance} for a proof and~\cite[Chapter~2]{takeuti-proof-theory} for general background). The ordinals that arise become much larger if we enrich the language: Let $\leq_1^+$ be defined just as $\leq_1$, but with $\mathcal L=\{\leq,\leq_1\}$ replaced by $\mathcal L_+=\{0,+\leq,\leq_1^+\}$, for a constant~$0$ and a ternary relation symbol~$+$ that is interpreted as the graph of ordinal addition. Now the minimal~$\alpha$ with $\alpha\leq_1^+\beta$ for all $\beta\geq\alpha$ is equal to the proof-theoretic ordinal of $\Pi^1_1\textsf{-CA}_0$ (as announced in~\cite{carlson-resemblance} and proved in~\cite{wilken-skolem-hull}), which dwarfs~$\varepsilon_0$. Carlson himself has suggested (see~\cite[Section~1]{carlson-resemblance} and also~\cite[Section~4.5]{montalban-open-problems}) to consider much more general extensions of the language that arise from dilators. This will be done in the present paper.

In the following we present constructions that are needed to explain our main result. The discussion will remain on a somewhat informal level, with full details deferred to Section~\ref{sect:relativize-patterns} below. Let us consider the category of ordinals (still identified with their sets of predecessors) and strictly increasing functions between them. Dilators, as defined by J.-Y.~Girard~\cite{girard-pi2}, are functors from ordinals to ordinals that preserve direct limits and pullbacks. These conditions ensure that each dilator~$D$ is determined by a set, namely, by its restriction to the subcategory of finite ordinals. More specifically, each ordinal $\gamma<D(\alpha)$ has a unique representation of the form $(\sigma;\gamma_0,\dots,\gamma_{n-1};\alpha)_D$ with $\gamma_0<\dots<\gamma_{n-1}<\alpha$, where $\sigma\in D(n)$ can be seen as a constructor symbol that takes $\gamma_0,\dots,\gamma_{n-1}$ and $\alpha$ as arguments (but not all $\sigma\in D(n)$ are allowed as constructors, cf.~Section~\ref{sect:relativize-patterns}). In general, the representation depends on~$\alpha$: for ordinals $\alpha<\beta$, the same $\gamma<D(\alpha)\leq D(\beta)$ can have entirely different representations with respect to~$\alpha$ and to~$\beta$. The reason is that the inclusion $\iota:\alpha\hookrightarrow\beta$ may induce a function $D(\iota):D(\alpha)\to D(\beta)$ that is not an inclusion itself (i.\,e.,~the range of~$D(\iota)$ need not be an initial segment of~$D(\beta)$). On the other hand, the representations are independent of~$\alpha$ when $D$ preserves inclusions. Dilators with this property are called flowers by Girard. We will work with a slightly more restrictive notion called normal dilator (previously studied in~\cite{freund-rathjen_derivatives}), which blends Girard's flowers with P.~Aczel's normal functors~\cite{aczel-phd,aczel-normal-functors}. If~$D$ is a normal dilator, then $\alpha\mapsto D(\alpha)$ is a normal function in the usual sense, i.\,e., it is strictly increasing and we have $D(\lambda)=\sup_{\alpha<\lambda}D(\alpha)$ whenever $\lambda$ is a limit ordinal. More importantly, the aforementioned representations become independent of the last component: assuming that~$D$ is normal, we write
\begin{equation}\label{eq:representations}
\gamma\simeq(\sigma;\gamma_0,\dots,\gamma_{n-1})_D
\end{equation}
if $\gamma$ has representation $(\sigma;\gamma_0,\dots,\gamma_{n-1};\alpha)_D$ for some (equivalently: for every) ordinal~$\alpha$ with $\gamma<D(\alpha)$. Now consider~(\ref{eq:representations}) as an $(n+1)$-ary relation in~$\gamma_0,\dots,\gamma_{n-1}$ and $\gamma$. Let~$\mathcal L_D$ be the language that contains a relation symbol for each such relation (i.\,e.,~for each constructor symbol associated with~$D$), as well as two binary relation symbols $\leq$ and~$\leq_1^D$. The interpretation of $\leq_1^D$ is defined as the interpretation of~$\leq_1$ above, but with $\mathcal L_D$ at the place of~$\mathcal L$. We now state our main result:

\begin{theorem}\label{thm:pattern-admissibles}
The following are equivalent over~$\textsf{ATR}_0^{\textsf{set}}$:
\begin{enumerate}[label=(\roman*)]
\item for any normal dilator~$D$, there is an ordinal~$\Omega$ with $\Omega\leq_1^D D(\Omega+1)$,
\item every set is contained in an admissible set.
\end{enumerate}
\end{theorem}

The theory $\textsf{ATR}_0^{\textsf{set}}$ is a set-theoretic version of $\textsf{ATR}_0$ (one of the central systems from reverse mathematics), over which it is conservative (due to Simpson~\cite{simpson82,simpson09}). We declare that $\textsf{ATR}_0^{\textsf{set}}$ contains the axiom of countability (as in~\cite[\mbox{Section~VII.3}]{simpson09}, while this axiom is marked as ``optional" in~\cite{simpson82}). Concerning statement~(ii), we recall that admissible sets are defined as transitive models of Kripke-Platek set theory (see~\cite{barwise-admissible}). Over~$\textsf{ATR}_0^{\textsf{set}}$, statement~(ii) is equivalent to $\Pi^1_1$-comprehension, another important principle of reverse mathematics (by~\cite[Section~7]{jaeger-admissibles} in conjunction with~\cite[Section~1.4]{freund-thesis}). This shows that our base theory is weak enough to make the equivalence informative (since $\Pi^1_1$-comprehension is known to be unprovable in~$\textsf{ATR}_0$). It also reveals that Theorem~\ref{thm:pattern-admissibles} answers Question~27 from A.~Montalb\'an's list~\cite{montalban-open-problems}, which asks for an equivalence between statement~(i) and $\Pi^1_1$-comprehension.

The motivation for our choice of base theory will become fully transparent in Section~\ref{sect:relativize-patterns}. For now, we just stress that $\textsf{ATR}_0^{\textsf{set}}$ is a set theory. This allows us to work with the ``semantic" definition of~$\leq_1^D$ that was given above. Patterns of resemblance can also be approached in a more ``syntactic" way (cf.~\cite{wilken-assignments,wilken-carlson-normal-form}). Based on such an approach, one may be able to prove a variant of Theorem~\ref{thm:pattern-admissibles} with a much weaker base theory in the language of second order arithmetic. This would certainly be of interest. For the present paper, we have decided that the elegance of the semantic definition is more important than an optimal base theory.

To prove one direction of Theorem~\ref{thm:pattern-admissibles}, we will show that~(i) holds when~$\Omega$ is the height of a suitable admissible set, as provided by~(ii). Details for this direction are given in Section~\ref{sect:admissible-to-elementarity}. The converse (and probably more surprising) direction from~(i) to~(ii) relies on a previous result of the author~\cite{freund-equivalence,freund-categorical,freund-computable}. For this result it is crucial to consider dilators that are not normal. If $D$ is such a dilator, there may not be any fixed point~$\alpha=D(\alpha)$. The best we can hope for is an ``almost" order preserving function $\vartheta:D(\alpha)\to\alpha$. To make this precise, we recall that each $\gamma<D(\alpha)$ has a unique representation $(\sigma;\gamma_0,\dots,\gamma_{n-1};\alpha)_D$ with $\gamma_0<\dots<\gamma_{n-1}<\alpha$. We will write $\supp_\alpha(\gamma)=\{\gamma_0,\dots,\gamma_{n-1}\}$. Now a function $\vartheta:D(\alpha)\to\alpha$ is called a Bachmann-Howard collapse if the following holds for all~$\gamma,\delta<D(\alpha)$:
\begin{itemize}[label={--}]
\item if we have $\gamma<\delta$ and $\supp_\alpha(\gamma)\subseteq\vartheta(\delta)$, then we have $\vartheta(\gamma)<\vartheta(\delta)$,
\item we have $\supp_\alpha(\gamma)\subseteq\vartheta(\gamma)$.
\end{itemize}
If such a collapse exists, then $\alpha$ is called a Bachmann-Howard fixed point of~$D$. This notion can be seen as a relativization of the Bachmann-Howard ordinal, which plays an important role in ordinal analysis (see e.\,g.~\cite{jaeger-kripke-platek,rathjen-weiermann-kruskal,FRW-kruskal}). In~\cite{freund-equivalence} it has been shown that statement~(ii) from Theorem~\ref{thm:pattern-admissibles} is equivalent to the assertion that every dilator has a Bachmann-Howard fixed point. We will show that this assertion follows from statement~(i), so that the latter implies~(ii). Let us sketch the argument, which will be worked out in Section~\ref{sect:elementarity-to-Bachmann-Howard}: The first step is to transform~$D$ into a normal dilator~$\Sigma D$, which can be characterized by $\Sigma D(\alpha):=\Sigma_{\gamma<\alpha}D(\gamma)$. Invoking statement~(i), we now fix an ordinal number $\Omega\leq_1^{\Sigma D}\Sigma D(\Omega+1)$. We get an embedding $\xi:D(\Omega)\to\Sigma D(\Omega+1)$ by setting $\xi(\gamma):=\Sigma D(\Omega)+\gamma$. It is not hard to see that $\Omega\leq_1^{\Sigma D}\Sigma D(\Omega+1)$ yields $\Omega\leq_1^{\Sigma D}\xi(\gamma)$, for an arbitrary~$\gamma<D(\Omega)$. Assume that we have $\xi(\gamma)\simeq(\sigma;\gamma_0,\dots,\gamma_{n-1})_{\Sigma D}$. We will see that $\Sigma D(\Omega)\leq\xi(\gamma)<\Sigma D(\Omega+1)$ entails $n>0$ and $\gamma_{n-1}=\Omega$. Hence $\eta=\Omega$ witnesses that the $\Sigma_1$-formula
\begin{equation*}
\exists\eta\,[(\gamma_0<\eta\land\ldots\land\gamma_{n-2}<\eta)\land\eta\leq_1^{\Sigma D}(\sigma;\gamma_0,\dots,\gamma_{n-2},\eta)_{\Sigma D}]
\end{equation*}
holds in $\Sigma D(\Omega+1)$. Due to $\Omega\leq_1^{\Sigma D}\Sigma D(\Omega+1)$, the same formula must hold in~$\Omega$. Still assuming $\xi(\gamma)\simeq(\sigma;\gamma_0,\dots,\gamma_{n-1})_{\Sigma D}$, this allows us to set
\begin{equation*}
\vartheta(\gamma):=\min\{\eta<\Omega\,|\,\{\gamma_0,\dots,\gamma_{n-2}\}\subseteq\eta\text{ and }\eta\leq_1^{\Sigma D}(\sigma;\gamma_0,\dots,\gamma_{n-2},\eta)_{\Sigma D}\}.
\end{equation*}
We will see that this defines a Bachmann-Howard collapse~$\vartheta:D(\Omega)\to\Omega$, as needed to derive~(ii) via the result from~\cite{freund-equivalence}. The construction of Bachmann-Howard fixed points via $\Sigma_1$-elementarity is interesting in its own right: it sheds further light on the collapsing construction that is so central to ordinal analysis.

\section{Normal dilators and patterns of resemblance}\label{sect:relativize-patterns}

In this section we recall the notion of (normal) dilator and its formalization in weak set theories. We then show how patterns of resemblance can be relativized to a normal dilator, making the discussion from the introduction precise.

Unless otherwise noted, the base theory for all definitions and results is~$\textsf{ATR}_0^{\operatorname{set}}$. Let us point out that Simpson gives two somewhat different but equivalent axiomatizations in~\cite{simpson82} and~\cite[Section~VII.3]{simpson09} (see the comparison in~\cite[Section~1.4]{freund-thesis}). For our purpose, the following facts will be central: First, $\textsf{ATR}_0^{\textsf{set}}$ includes axiom beta, which states that any well-founded relation can be collapsed to set-membership. In our context, this will ensure that the values of dilators exist as actual ordinals (which arise by collapsing certain term representation systems, cf.~Definition~\ref{def:dil-ordinals}). Secondly, the theory $\mathsf{ATR}_0^{\textsf{set}}$ shows that all primitive recursive set functions (in the sense of R.~Jensen and C.~Karp~\cite{jensen-karp}) are total. This will, in particular, allow us to define the restriction of~$\leq_1^D$ to $\beta\times\beta$ by recursion on the ordinal~$\beta$. Finally, as agreed in the introduction, we work with the version of $\mathsf{ATR}_0^{\textsf{set}}$ that includes the axiom of countability. This axiom asserts that any set admits an injection into the finite ordinals (equivalently, a surjection in the converse direction, provided we are concerned with a non-empty set). It ensures that statement~(ii) in Theorem~\ref{thm:pattern-admissibles} is equivalent to $\Pi^1_1$-comprehension over the natural numbers. Even more importantly, it is used in the prove of a previous result~\cite{freund-equivalence}, to which Theorem~\ref{thm:pattern-admissibles} will be reduced. While the countability assumption may be unusual in the context of set theory, it is very natural from the viewpoint of reverse mathematics.

Let us write $\ord$ for the category of ordinals (identified with their sets of predecessors) and strictly increasing functions between them. By $\nat$ we denote the full subcategory of finite ordinals (natural numbers). We will write $[\cdot]^{<\omega}$ for the finite subset functor on the category of sets, with
\begin{align*}
[X]^{<\omega}&:=\text{the set of finite subsets of~$X$},\\
[f]^{<\omega}(a)&:=\{f(x)\,|\,x\in X\}\quad\text{(for $f:X\to Y$ and $a\in[X]^{<\omega}$)}.
\end{align*}
We will also apply~$[\cdot]^{<\omega}$ to ordinals, omitting the forgetul functor to the category of sets. Conversely, we often assume that sets of ordinals are equipped with the usual order. For $n\in\nat$, the set $[n]^{<\omega}$ is, of course, the full powerset of $n=\{0,\dots,n-1\}$. The following definition does not quite coincide with the ones in~\cite[Section~4.4]{girard-pi2} and~\cite[Section~2]{freund-computable}. However, the resulting notion of dilator will be equivalent (cf.~the proof of Theorem~\ref{thm:from-freund-equivalence}). Note that $\rng(g)$ denotes the range (image) of~$g$.

\begin{definition}\label{def:pre-dilator}
A pre-dilator consists of a functor $D:\nat\to\ord$ and a natural transformation $\supp:D\Rightarrow[\cdot]^{<\omega}$, such that
\begin{equation}\tag{``support condition"}
\supp_n(\sigma)\subseteq\rng(f)\quad\Rightarrow\quad\sigma\in\rng(D(f))
\end{equation}
holds for any morphism $f:m\to n$ in $\nat$ and any element $\sigma\in D(n)$.
\end{definition}

In~\cite[Definition~2.1]{freund-computable}, the support condition was only required for the unique morphism $f_\sigma:m_\sigma\to n$ with range $\supp_n(\sigma)$, where $m_\sigma$ is determined as the cardinality of that set. This does not make a difference, as any morphism $f:m\to n$ with $\supp_n(\sigma)\subseteq\rng(f)$ allows for a factorization $f_\sigma=f\circ g$ with $g:m_\sigma\to m$. Also note that the converse of the support condition follows from naturality, namely
\begin{equation*}
\supp_n(D(f)(\sigma_0))=[f]^{<\omega}(\supp_m(\sigma_0))\subseteq\rng(f).
\end{equation*}
In second-order set theory we would be able to define class-sized dilators as functors from ordinals to ordinals that preserve pullbacks and direct limits. The last two conditions are equivalent to the existence of (necessarily unique) supports as in Definition~\ref{def:pre-dilator} (see~\cite[Remark~2.2.2]{freund-thesis}). Hence the restriction of a class-sized dilator to~$\nat$ is a pre-dilator in the sense of Definition~\ref{def:pre-dilator}. To get back the class-sized dilator from its restriction, it suffices to extend the latter by direct limits (see~\cite[\mbox{Proposition~2.1}]{freund-computable}). Indeed, any pre-dilator in the sense of Definition~\ref{def:pre-dilator} can be extended in this way. The extension will automatically preserve pullbacks and direct limits, but in general it will be a functor into linear orders rather than ordinals. On the other hand, the requirement that all values are well-founded (and hence isomorphic to unique ordinals) can be expressed in the usual first-order language of set theory. This shows that the second-order viewpoint is not required after all. To describe the extension of pre-dilators in our base theory, we will use the following.

\begin{definition}
The trace of a pre-dilator~$D=(D,\supp)$ is given by
\begin{equation*}
\tr(D)=\{(\sigma,n)\,|\,n\in\nat\text{ and }\sigma\in D(n)\text{ with }\supp_n(\sigma)=n\}.
\end{equation*}
\end{definition}

Note that $\nat$ is a small category equivalent to the large (but locally small) category of all finite linear orders. To make the equivalence explicit, we write $|a|$ and $\en_a:|a|\to a$ for the cardinality and the strictly increasing enumeration of a finite order~$a$. If $f:a\to b$ is a strictly increasing function between finite orders, we write $|f|:|a|\to|b|$ for the unique morphism in~$\nat$ that satisfies $\en_b\circ|f|=f\circ\en_a$. Given an order~$X$ and a suborder~$Y$, we will write $\iota_X^Y:X\hookrightarrow Y$ for the inclusion. The following coincides with~\cite[Definition~2.2]{freund-computable}.

\begin{definition}\label{def:dil-extend}
Given a pre-dilator~$D$ and an ordinal~$\alpha$, we define $\overline D(\alpha)$ as the set of all expressions $(\sigma;\gamma_0,\dots,\gamma_{n-1};\alpha)_D$ for an element $(\sigma,n)\in\tr(D)$ and ordinals $\gamma_0<\dots<\gamma_{n-1}<\alpha$. To get a binary relation $<_{\overline D(\alpha)}$ on $\overline D(\alpha)$, we declare that
\begin{equation*}
(\sigma;\gamma_0,\dots,\gamma_{m-1};\alpha)_D<_{\overline D(\alpha)}(\tau;\delta_0,\dots,\delta_{n-1};\alpha)_D
\end{equation*}
is equivalent to $D(|\iota_c^{c\cup d}|)(\sigma)<_{D(|c\cup d|)}D(|\iota_d^{c\cup d}|)(\tau)$ with $c=\{\gamma_0,\dots,\gamma_{m-1}\}$ and $d=\{\delta_0,\dots,\delta_{n-1}\}$ (where $c\cup d$ carries the usual order between ordinals). For a morphism $f:\alpha\to\beta$ of~$\ord$, we set
\begin{equation*}
\overline D(f)\left((\sigma;\gamma_0,\dots,\gamma_{n-1};\alpha)_D\right):=(\sigma;f(\gamma_0),\dots,f(\gamma_{n-1});\beta)_D
\end{equation*}
to define a function $\overline D(f):\overline D(\alpha)\to\overline D(\beta)$.
\end{definition}

One can verify that $\overline D$ is a functor from ordinals to linear orders (with strictly increasing functions as morphisms, cf.~\cite[Lemma~2.2]{freund-computable}). Also,
\begin{equation*}
\{\delta_0,\dots,\delta_{n-1}\}\subseteq\rng(f)\quad\Rightarrow\quad(\sigma;\delta_0,\dots,\delta_{n-1};\beta)_D\in\rng(\overline D(f))
\end{equation*}
holds for any morphism $f:\alpha\to\beta$ and any element $(\sigma;\delta_0,\dots,\delta_{n-1};\beta)_D$ of~$\overline D(\beta)$. Hence $\{\delta_0,\dots,\delta_{n-1}\}$ can be seen as the support of the given element.

\begin{definition}\label{def:dilator}
A pre-dilator~$D$ is called a dilator if the linear order~$\overline D(\alpha)$ is well founded for every ordinal~$\alpha$.
\end{definition}

The following construction is possible because our base theory~$\textsf{ATR}_0^{\textsf{set}}$ includes axiom beta, which allows us to collapse well orders to ordinals. Let us point out that the action of a dilator~$D$ on objects and morphisms of~$\nat$ is already defined. The following does not lead to any conflict, since~\cite[Lemma~2.6]{freund-rathjen_derivatives} provides a family of isomorphisms $D(n)\cong\overline D(n)$ that is natural in~$n\in\nat$ and respects supports.

\begin{definition}\label{def:dil-ordinals}
Assume that~$D$ is a dilator. For each ordinal number $\alpha$, let $D(\alpha)$ and $\eta_\alpha:D(\alpha)\to\overline D(\alpha)$ be unique such that $D(\alpha)$ is an ordinal and $\eta_\alpha$ is an order isomorphism. Given a morphism $f:\alpha\to\beta$ in~$\ord$, we define $D(f):D(\alpha)\to D(\beta)$ as the unique function with $\eta_\beta\circ D(f)=\overline D(f)\circ\eta_\alpha$. By
\begin{equation*}
\supp_\alpha(\gamma):=\{\gamma_0,\dots,\gamma_{n-1}\}\quad\text{for $\eta_\alpha(\gamma)=(\sigma;\gamma_0,\dots,\gamma_{n-1};\alpha)_D$}
\end{equation*}
we define a family of functions~$\supp_\alpha:D(\alpha)\to[\alpha]^{<\omega}$.
\end{definition}

We note that $D:\ord\to\ord$ is a functor equivalent to~$\overline D$, and in fact a class-sized dilator in the sense discussed above. From a foundational viewpoint there is, nevertheless, an important difference: the map $\alpha\mapsto\overline D(\alpha)$ is a primitive recursive set function while $\alpha\mapsto D(\alpha)$, in general, is not (since primitive recursive set functions cannot collapse arbitrary well orders, cf.~\cite[Remark~2.3.7]{freund-computable}). We have mentioned that the present definition of pre-dilators is slightly different from the one in~\cite{freund-computable}. Let us verify that the following crucial result remains valid:

\begin{theorem}[{\cite{freund-equivalence,freund-computable}}]\label{thm:from-freund-equivalence}
The following are equivalent over $\textsf{ATR}_0^{\textsf{set}}$:
\begin{enumerate}[label=(\roman*)]
\item for any dilator~$D$ there is an~$\alpha\in\ord$ and a function $\vartheta:D(\alpha)\to\alpha$ such~that
\begin{enumerate}[label=(\alph*)]
\item if $\gamma<\delta<D(\alpha)$ and $\supp_\alpha(\gamma)\subseteq\vartheta(\delta)$, then $\vartheta(\gamma)<\vartheta(\delta)$,
\item we have $\supp_\alpha(\gamma)\subseteq\vartheta(\gamma)$ for all $\gamma<D(\alpha)$,
\end{enumerate}
\item every set is contained in an admissible set.
\end{enumerate}
\end{theorem}
A function~$\vartheta$ as in statement~(i) is called a Bachmann-Howard collapse. If such a function exists, then~$\alpha$ is called a Bachmann-Howard fixed point of~$D$.
\begin{proof}
If we were to replace ``dilator" (in the sense of Definitions~\ref{def:pre-dilator} and~\ref{def:dilator} above) by ``set-sized dilator" (in the sense of~\cite[Definitions~2.1 and~2.3]{freund-computable}), then the result would hold by~\cite[Theorem~9.7]{freund-equivalence} and~\cite[Proposition~2.2]{freund-computable}. The only difference between the two notions is that the values of a set-sized dilator may be arbitrary well orders rather than ordinals. This makes the notion more general in the absence of axiom beta. Since the latter is included in our base theory~$\mathsf{ATR}_0^{\textsf{set}}$, the difference vanishes. More precisely, axiom beta allows us to transform a given set-sized dilator into an isomorphic dilator in the sense of the present paper. Also, if statement~(i) holds for a dilator~$D$, then it holds for any (set-sized) dilator that is isomorphic to~$D$, as verified in the proofs of~\cite[Proposition~2.2 and Lemma~2.5]{freund-computable}.
\end{proof}

For $\alpha<\beta$, the inclusion $\iota_\alpha^\beta:\alpha\hookrightarrow\beta$ induces a morphism $D(\iota_\alpha^\beta):D(\alpha)\to D(\beta)$. In general, the latter will not coincide with the inclusion $D(\alpha)\hookrightarrow D(\beta)$ (since the range of $\overline D(\iota_\alpha^\beta):\overline D(\alpha)\to\overline D(\beta)$ need not be an initial segment of~$\overline D(\beta)$, even though $\overline D(\iota_\alpha^\beta)$ is an inclusion by construction). When $D(\iota_\alpha^\beta)$ is not the inclusion, there is no obvious syntactical relation between the expressions $\eta_\alpha(\gamma)\in\overline D(\alpha)$ and $\eta_\beta(\gamma)\in\overline D(\beta)$ that represent the same ordinal~$\gamma<D(\alpha)\leq D(\beta)$. This turns out to be inconvenient in the context of patterns of resemblance. We will see that the issue vanishes for dilators with the following additional structure (cf.~\cite{freund-rathjen_derivatives}).

\begin{definition}
A normal (pre-)dilator consists of a (pre-)dilator~$(D,\supp)$ and a natural transformation $\mu:I\Rightarrow D$ (for the inclusion functor $I:\nat\to\ord$) with
\begin{equation*}\tag{``normality condition"}
\sigma<\mu_n(k)\quad\Leftrightarrow\quad\supp_n(\sigma)\subseteq k=\{0,\dots,k-1\}
\end{equation*}
for any $k<n\in\nat$ and $\sigma\in D(n)$.
\end{definition}

The reader may notice that all values $\mu_n(k)$ are determined by $\mu_1(0)$ (apply naturality to $\iota:1\to n$ with $\iota(0)=k$). Nevertheless, one needs to require the normality condition for all~$n\in\nat$. By combining the latter with the support condition, one learns that $D(\iota_k^n):D(k)\to D(n)$ has range~$\{\sigma\in D(n)\,|\,\sigma<\mu_n(k)\}$, which is an initial segment. We now extend this property from~$\nat$ to~$\ord$. First note that we have $\supp_1(\mu_1(0))=\{0\}=1$, since $\supp_1(\mu_1(0))=\emptyset=0$ would lead to $\mu_1(0)<\mu_1(0)$, by the normality condition. This yields $(\mu_1(0),1)\in\tr(D)$, as needed to justify the following. Concerning the last sentence of the definition, we point out that $\mu_\alpha$ is already defined for~$\alpha=n\in\nat$. No conflict arises, since the isomorphism $D(n)\cong\overline D(n)$ from~\cite[Lemma~2.6]{freund-rathjen_derivatives} sends $\mu_k(n)$ to $(\mu_1(0);k;n)_D$.

\begin{definition}\label{def:normal-extend}
Consider a normal pre-dilator~$D=(D,\mu)$. For each ordinal~$\alpha$ we define a function $\overline\mu_\alpha:\alpha\to\overline D(\alpha)$ by
\begin{equation*}
\overline\mu_\alpha(\gamma):=(\mu_1(0);\gamma;\alpha)_D.
\end{equation*}
If~$D$ is a dilator, we also define $\mu_\alpha:\alpha\to D(\alpha)$ by stipulating $\eta_\alpha\circ\mu_\alpha=\overline\mu_\alpha$.
\end{definition}

In~\cite[Proposition~2.13]{freund-rathjen_derivatives} it has been shown that $\overline\mu:I\Rightarrow\overline D$ is a natural transformation (for the inclusion~$I$ of ordinals into linear orders), and that we have
\begin{equation}\label{eq:normality-extended-overline}
(\sigma;\gamma_0,\dots,\gamma_{n-1};\alpha)_D<_{\overline D(\alpha)}\overline\mu_\alpha(\delta)\quad\Leftrightarrow\quad \gamma_i<_X \delta\text{ for all }i<n
\end{equation}
for all ordinals~$\delta<\alpha$ and any element $(\sigma;\gamma_0,\dots,\gamma_{n-1};\alpha)_D$ of $\overline D(\alpha)$. In the case where $D$ is a dilator, it follows that $\mu:I\to D$ is a natural transformation (now for the identity~$I:\ord\to\ord$) with
\begin{equation}\label{eq:normality-extended}
\gamma<\mu_\alpha(\delta)\quad\Leftrightarrow\quad\supp_\alpha(\gamma)\subseteq\delta
\end{equation}
for all $\gamma<D(\alpha)$ and $\delta<\alpha$. If $\alpha$ is a limit, this entails that $\{\mu_\alpha(\delta)\,|\,\delta<\alpha\}$ is cofinal in~$D(\alpha)$ (since the supports are finite). To conclude that $\alpha\mapsto D(\alpha)$ is a normal function in the usual sense (i.\,e., strictly increasing and continuous at limit ordinals), it is now enough to show that~$\alpha<\beta$ implies $D(\alpha)=\mu_\beta(\alpha)<D(\beta)$. Indeed, the support and normality conditions entail that the morphism $D(\iota_\alpha^\beta):D(\alpha)\to D(\beta)$ that arises from the inclusion $\iota_\alpha^\beta:\alpha\hookrightarrow\beta$ has range
\begin{equation*}
\rng(D(\iota_\alpha^\beta))=\{\gamma\in D(\beta)\,|\,\gamma<\mu_\beta(\alpha)\}.
\end{equation*}
Induction on~$\gamma<D(\alpha)$ yields $D(\iota_\alpha^\beta)(\gamma)=\gamma$ and hence $D(\alpha)=\mu_\beta(\alpha)$, as desired. For our purpose, the following consequence of normality is particularly important:

\begin{lemma}\label{lem:normal-representation-independent}
If $D$ is a normal dilator, then we have
\begin{equation*}
\eta_\alpha(\gamma)=(\sigma;\gamma_0,\dots,\gamma_{n-1};\alpha)_D\quad\Leftrightarrow\quad\eta_\beta(\gamma)=(\sigma;\gamma_0,\dots,\gamma_{n-1};\beta)_D
\end{equation*}
for arbitrary ordinals~$\alpha<\beta$ and any $\gamma<D(\alpha)$.
\end{lemma}
\begin{proof}
In view of Definition~\ref{def:dil-ordinals} we have $\eta_\beta\circ D(\iota_\alpha^\beta)=\overline D(\iota_\alpha^\beta)\circ\eta_\alpha$, where $\iota_\alpha^\beta$ is the inclusion. Above we have seen that $D(\iota_\alpha^\beta)(\gamma)=\gamma$ holds for~$\gamma<D(\alpha)$. Assuming the left part of the equivalence in the lemma, we thus get
\begin{align*}
\eta_\beta(\gamma)=\eta_\beta\circ D(\iota_\alpha^\beta)(\gamma)&=\overline D(\iota_\alpha^\beta)\circ\eta_\alpha(\gamma)=\overline D(\iota_\alpha^\beta)((\sigma;\gamma_0,\dots,\gamma_{n-1};\alpha)_D)=\\
{}&=(\sigma;\iota_\alpha^\beta(\gamma_0),\dots,\iota_\alpha^\beta(\gamma_{n-1});\beta)_D=(\sigma;\gamma_0,\dots,\gamma_{n-1};\beta)_D,
\end{align*}
where the penultimate equality relies on Definition~\ref{def:dil-extend}. The converse implication follows immediately, as its failure would lead to two different values for~$\eta_\beta(\gamma)$.
\end{proof}

Motivated by the lemma, we introduce the following terminology.

\begin{definition}\label{def:representation}
Consider a normal dilator~$D$. If $\eta_\alpha(\gamma)=(\sigma;\gamma_0,\dots,\gamma_{n-1};\alpha)_D$ holds for some (or equivalently every) ordinal~$\alpha$ with $\gamma<D(\alpha)$, then we write
\begin{equation*}
\gamma\simeq(\sigma;\gamma_0,\dots,\gamma_{n-1})_D
\end{equation*}
and call the right side the representation of~$\gamma$.
\end{definition}

Speaking of \emph{the} representation is justified in view of the following.

\begin{proposition}\label{prop:representations-exist-unique}
Assume that~$D$ is a normal dilator. Then any ordinal~$\gamma$ has a unique representation
\begin{equation*}
\gamma\simeq(\sigma;\gamma_0,\dots,\gamma_{n-1})_D,
\end{equation*}
with $(\sigma,n)\in\tr(D)$ and $\gamma_0<\dots<\gamma_{n-1}$. In terms of this representation, we have
\begin{equation*}
\gamma<D(\alpha)\quad\Leftrightarrow\quad n=0\text{ or }\gamma_{n-1}<\alpha
\end{equation*}
for any ordinal~$\alpha$.
\end{proposition}
\begin{proof}
We have already seen that $\alpha\mapsto D(\alpha)$ is strictly increasing and hence unbounded. To prove existence, pick an ordinal~$\alpha$ with $\gamma<D(\alpha)$. The value $\eta_\alpha(\gamma)$ is an element of~$\overline D(\alpha)$ and hence of the form $(\sigma;\gamma_0,\dots,\gamma_{n-1};\alpha)_D$ with $(\sigma,n)\in\tr(D)$ and $\gamma_0<\dots<\gamma_{n-1}<\alpha$. It follows that~$\gamma$ has representation as in the proposition. To show uniqueness, consider a competitor $\gamma\simeq(\tau;\gamma'_0,\dots,\gamma'_{m-1})_D$ for the representation of~$\gamma$. Then $\eta_\beta(\gamma)=(\tau;\gamma'_0,\dots,\gamma'_{m-1};\beta)_D$ holds for some~$\beta$ with $\gamma<D(\beta)$. By Lemma~\ref{lem:normal-representation-independent} we get
\begin{equation*}
(\tau;\gamma'_0,\dots,\gamma'_{m-1};\beta)_D=(\sigma;\gamma_0,\dots,\gamma_{n-1};\beta)_D,
\end{equation*}
so that the two representations of~$\gamma$ coincide after all. The direction ``$\Rightarrow$" of the equivalence in the proposition was shown in our proof of existence. For the converse implication, write $\eta_\beta(\gamma)=(\sigma;\gamma_0,\dots,\gamma_{n-1};\beta)_D$ with $\gamma<D(\beta)$ and $\alpha<\beta$. If the right side of the desired equivalence holds, then we have
\begin{equation*}
\supp_\beta(\gamma)=\{\gamma_0,\dots,\gamma_{n-1}\}\subseteq\alpha.
\end{equation*}
We now get $\gamma<\mu_\beta(\alpha)=D(\alpha)$ by~(\ref{eq:normality-extended}) from above.
\end{proof}

The definition of $\leq_1^D$ from the introduction can now be made official. Given a normal dilator~$D$, let $\mathcal L_D$ be the language that consists of two binary relation symbols~$\leq$ and $\leq_1^D$, as well as an $(n+1)$-ary relation symbol $\gamma\simeq(\sigma;\gamma_0,\dots,\gamma_{n-1})_D$ for each element $(\sigma,n)\in\tr(D)$ (considered as a relation in~$\gamma$ and $\gamma_0,\dots,\gamma_{n-1}$). We will only interpret $\mathcal L_D$ in structures that have a set of ordinals as universe. The symbol $\leq$ is always interpreted as the usual inequality between ordinals. Note that we do not need a symbol for equality, as $\alpha=\beta$ is equivalent to the conjunction of~$\alpha\leq\beta$ and $\beta\leq\alpha$. The relation symbols $\gamma\simeq(\sigma;\gamma_0,\dots,\gamma_{n-1})_D$ are always interpreted according to Definition~\ref{def:representation}. In particular, the given relation can only hold when we have~$\gamma_0<\dots<\gamma_{n-1}$. As usual, a bijection $f:X\to Y$ between $\mathcal L_D$-structures is an $\mathcal L_D$-isomorphism if we have
\begin{equation*}
\gamma\simeq(\sigma;\gamma_0,\dots,\gamma_{n-1})_D\quad\Leftrightarrow\quad f(\gamma)\simeq(\sigma;f(\gamma_0),\dots,f(\gamma_{n-1}))_D
\end{equation*}
for all $(\sigma,n)\in\tr(D)$ and $\gamma,\gamma_0,\dots,\gamma_{n-1}\in X$, and if analogous equivalences hold with respect to~$\leq$ and $\leq_1^D$. Concerning the interpretation of~$\leq_1^D$, we adopt the following as our official definition. In Proposition~\ref{prop:leq_1^D-Sigma_1-elem} below, we will show that it coincides with the more familiar formulation in terms of~$\Sigma_1$-elementarity. Note that this is not entirely obvious, as our language can be infinite (depending on~$D$).

\begin{definition}\label{def:leq_1^D}
Invoking recursion on~$\beta$, we declare that $\alpha\leq_1^D\beta$ is equivalent to the conjunction of~$\alpha\leq\beta$ and the following: for all finite sets $X\subseteq\alpha$ and $Y\subseteq\beta\backslash\alpha$, there is a finite $\widetilde Y\subseteq\alpha$ and an $\mathcal L_D$-isomorphism $f:X\cup Y\to X\cup\widetilde Y$ that fixes~$X$.
\end{definition}

Let us briefly explain how our base theory $\textsf{ATR}_0^{\textsf{set}}$ accommodates this definition: The idea is to define $\leq_1^D\!\restriction\!(\beta\times\beta)$ by primitive recursion (cf.~\cite{jensen-karp,simpson82}). This is obstructed by the fact that $\delta\mapsto D(\delta)$ may not be primitive recursive, as we have noted above. To resolve this issue, we restrict attention to ordinals~$\beta$ below a fixed (but of course arbitrary) bound~$\delta$. As we have seen in the context of Definition~\ref{def:dil-ordinals}, axiom beta allows us to consider the isomorphism $\eta_\delta:D(\delta)\to\overline D(\delta)$. Using the latter as a parameter, a primitive recursive set function can decide
\begin{equation*}
\gamma\simeq(\sigma;\gamma_0,\dots,\gamma_{n-1})_D
\end{equation*}
for all $\gamma$ and $\gamma_0,\dots,\gamma_{n-1}$ below~$\delta\leq D(\delta)$. Concerning the definition of~$\leq_1^D$, it follows that the recursion step for~$\beta<\delta$ is primitive recursive (still with $\eta_\delta$ as parameter). This shows that $\leq_1^D$ can be defined as a binary relation on the class of ordinals, and that each of the restrictions $\leq_1^D\!\restriction\!(\beta\times\beta)$ exists as a set. In the proof of Proposition~\ref{prop:clubs-Delta} we will see that the class function $(D,\beta)\mapsto{\leq_1^D\!\restriction\!(\beta\times\beta)}$ is $\Sigma$-definable and total in admissible sets.

As the final result of this section, we show that the given definition of~$\leq_1^D$ is equivalent to one in terms of $\Sigma_1$-elementarity. To be precise, we state that our \mbox{$\Sigma_1$-}formulas may begin with a string of existential quantifiers, followed by a quantifier-free formula. Each set of ordinals gives rise to a canonical $\mathcal L_D$-structure, in which the relation symbols are interpreted as specified above. The relation~$\vDash$ of satisfaction in an~$\mathcal L_D$-structure is readily formalized in terms of primitive recursive set functions (cf.~e.\,g.~\cite[Section~1.3]{freund-thesis}). In fact, the equivalence with~(ii) below is not needed for any of the technical arguments in this paper (and the proof of (i)$\Leftrightarrow$(iii) stands on its own). So if the reader is happy with Definition~\ref{def:leq_1^D} as characterization of~$\leq_1^D$, they may avoid the notion of satisfaction altogether. The equivalence with~(iii) shows that certain equivalences in the definition of \mbox{$\mathcal L_D$-}isomorphism can be reduced to implications (provided we quantify over all finite substructures).

\begin{proposition}\label{prop:leq_1^D-Sigma_1-elem}
The following are equivalent for each normal dilator~$D$ and all ordinal numbers $\alpha\leq\beta$:
\begin{enumerate}[label=(\roman*)]
\item we have $\alpha\leq_1^D\beta$ (according to Definition~\ref{def:leq_1^D}),
\item the $\mathcal L_D$-structure~$\alpha$ is a $\Sigma_1$-elementary substructure of~$\beta$, i.\,e., we have
\begin{equation*}
\beta\vDash\varphi(\alpha_1,\dots,\alpha_n)\quad\Rightarrow\quad\alpha\vDash\varphi(\alpha_1,\dots,\alpha_n)
\end{equation*}
for any $\Sigma_1$-formula $\varphi$ of $\mathcal L_D$ and all parameters $\alpha_1,\dots,\alpha_n<\alpha$,
\item for all finite sets $X\subseteq\alpha$ and $Y\subseteq\beta\backslash\alpha$, there is a finite $\widetilde Y\subseteq\alpha$ and an $\{\leq,\leq_1^D\}$-isomorphism $f:X\cup Y\to X\cup\widetilde Y$ that fixes~$X$, such that we have
\begin{equation*}
\gamma\simeq(\sigma;\gamma_0,\dots,\gamma_{n-1})_D\quad\Rightarrow\quad f(\gamma)\simeq(\sigma;f(\gamma_0),\dots,f(\gamma_{n-1}))_D
\end{equation*}
for all $(\sigma,n)\in\tr(D)$ and $\gamma,\gamma_0,\dots,\gamma_{n-1}\in X\cup Y$.
\end{enumerate}
\end{proposition}
\begin{proof}
Concerning~(i)$\Rightarrow$(ii), the premise of the implication in~(ii) entails that we have $Z\vDash\varphi(\alpha_1,\dots,\alpha_n)$ for some finite subset $Z\supseteq\{\alpha_1,\dots,\alpha_n\}$ of~$\beta$, since $\varphi$ is existential and the language contains no function symbols. Write $Z=X\cup Y$ with $X\subseteq\alpha$ and $Y\subseteq\beta\backslash\alpha$. From~(i) we get a finite set $\widetilde Y\subseteq\alpha$ and an \mbox{$\mathcal L_D$-isomorphism} $f:X\cup Y\to X\cup\widetilde Y$ that fixes~$X$. This yields $X\cup\widetilde Y\vDash\varphi(f(\alpha_1),\dots,f(\alpha_n))$ and then $\alpha\vDash\varphi(\alpha_1,\dots,\alpha_n)$, as $f$ fixes~$X=Z\cap\alpha\supseteq\{\alpha_1,\dots,\alpha_n\}$ and as $\Sigma_1$-formulas are preserved upwards. To show that~(ii) implies~(iii), we consider arbitrary sets $X=\{\alpha_1,\dots,\alpha_m\}\subseteq\alpha$ and $Y=\{\beta_1,\dots,\beta_n\}\subseteq\beta\backslash\alpha$. The restrictions of~$\leq$ and $\leq_1^D$ to $X\cup Y$ can be fully described by a quantifier-free $\mathcal L_D$-formula $\theta_0$, in the sense that $\theta_0(\gamma_1,\dots,\gamma_m,\delta_1,\dots,\delta_n)$ holds precisely when $f(\alpha_i)=\gamma_i$ and $f(\beta_j)=\delta_j$ determines a $\{\leq,\leq_1^D\}$-isomorphism~$f:X\cup Y\to\rng(f)$. Furthermore, there are only finitely many true statements
\begin{equation*}
\zeta\simeq(\sigma;\zeta_0,\dots,\zeta_{k-1})_D
\end{equation*}
with $\{\zeta\}\cup\{\zeta_0,\dots,\zeta_{k-1}\}\subseteq X\cup Y$, by the uniqueness part of Proposition~\ref{prop:representations-exist-unique}. It is straightforward to specify a quantifier-free $\mathcal L_D$-formula $\theta_1$ that guarantees these statements, in the sense that we have $\theta_1(\gamma_1,\dots,\gamma_m,\delta_1,\dots,\delta_n)$ precisely when the implications in~(iii) hold for $f:X\cup Y\to\rng(f)$ with $f(\alpha_i)=\gamma_i$ and $f(\beta_j)=\delta_j$. Now let $\varphi(\alpha_1,\dots,\alpha_m)$ be the $\Sigma_1$-formula
\begin{equation*}
\exists y_1,\dots,y_n[\theta_0(\alpha_1,\dots,\alpha_m,y_1,\dots,y_n)\land\theta_1(\alpha_1,\dots,\alpha_m,y_1,\dots,y_n)].
\end{equation*}
The assignment $y_i:=\beta_i$ witnesses $\beta\vDash\varphi(\alpha_1,\dots,\alpha_m)$. Invoking~(ii), we learn that there are ordinals~$\delta_1,\dots,\delta_n<\alpha$ such that we have
\begin{equation*}
\theta_0(\alpha_1,\dots,\alpha_m,\delta_1,\dots,\delta_n)\land\theta_1(\alpha_1,\dots,\alpha_m,\delta_1,\dots,\delta_n).
\end{equation*}
Set $\widetilde Y:=\{\delta_1,\dots,\delta_n\}$, and define $f:X\cup Y\to X\cup\widetilde Y$ by $f(\alpha_i):=\alpha_i$ and~$f(\beta_j):=\delta_j$. By construction, $f$ is an $\{\leq,\leq_1^D\}$-isomorphism, fixes the set $X=\{\alpha_1,\dots,\alpha_m\}$, and validates the implications in~(iii). Finally, we show that~(iii) implies~(i). For the duration of this argument, let us say that a set $Z\subseteq\beta$ is closed if we have
\begin{equation*}
Z\ni\gamma\simeq(\sigma;\gamma_0,\dots,\gamma_{n-1})_D\quad\Rightarrow\quad\{\gamma_0,\dots,\gamma_{n-1}\}\subseteq Z,
\end{equation*}
for all $\gamma<\beta$. If $X\cup Y$ is closed, then the implications in~(iii) will automatically upgrade to equivalences. Indeed, assume that we have
\begin{equation*}
f(\gamma)\simeq(\sigma;f(\gamma_0),\dots,f(\gamma_{n-1}))_D
\end{equation*}
with~$\gamma\in X\cup Y$. By the existence part of Proposition~\ref{prop:representations-exist-unique}, we obtain a representation $\gamma\simeq(\tau;\gamma'_0,\dots,\gamma'_{m-1})_D$. As~$X\cup Y$ is closed, we get $\{\gamma'_0,\dots,\gamma'_{m-1}\}\subseteq X\cup Y$. Hence one of the implications in~(iii) yields
\begin{equation*}
f(\gamma)\simeq(\tau;f(\gamma'_0),\dots,f(\gamma'_{m-1}))_D.
\end{equation*}
Invoking the uniqueness part of Proposition~\ref{prop:representations-exist-unique}, we can now conclude that we have $\sigma=\tau$ and $f(\gamma_i)=f(\gamma'_i)$ for all~$i<m=n$. Given that $f$ is an $\{\leq,\leq_1^D\}$-isomorphism and hence injective, we obtain $\gamma_i=\gamma'_i$ and thus
\begin{equation*}
\gamma\simeq(\tau;\gamma'_0,\dots,\gamma'_{m-1})_D=(\sigma;\gamma_0,\dots,\gamma_{n-1})_D,
\end{equation*}
as needed for the converse of our implication from~(iii). As preparation for the final part of the argument, we show that any finite set $Z\subseteq\beta$ is contained in a finite closed set~$\cl(Z)\subseteq\beta$. In view of $\gamma<D(\gamma+1)$ we get
\begin{equation*}
\gamma\simeq(\sigma;\gamma_0,\dots,\gamma_{n-1})_D\quad\Rightarrow\quad\{\gamma_0,\dots,\gamma_{n-1}\}\subseteq\gamma+1,
\end{equation*}
using the equivalence from Proposition~\ref{prop:representations-exist-unique}. By recursion on~$\gamma<\beta$ we now define
\begin{equation*}
\cls(\gamma):=\{\gamma\}\cup\bigcup\{\cls(\gamma_i)\,|\,i<n\text{ and }\gamma_i<\gamma\}\quad\text{for }\gamma\simeq(\sigma;\gamma_0,\dots,\gamma_{n-1})_D.
\end{equation*}
This amounts to a primitive recursion with $\eta_\beta:D(\beta)\to\overline D(\beta)$ as parameter, where the latter allows us to compute the representations of ordinals $\gamma<\beta\leq D(\beta)$ (cf.~the paragraph after Definition~\ref{def:leq_1^D}). A straightforward induction on~$\gamma$ shows that $\cls(\gamma)\subseteq\gamma+1$ is finite and closed. The desired closure of a finite set $Z\subseteq\beta$ can thus be given by $\cl(Z):=\bigcup\{\cls(\gamma)\,|\,\gamma\in Z\}\subseteq\beta$. We now have all ingredients to deduce~(i) from~(iii). In view of Definition~\ref{def:leq_1^D}, we consider finite sets~$X\subseteq\alpha$ and~$Y\subseteq\beta\backslash\alpha$. Write $\cl(X\cup Y)=X'\cup Y'$ with $X\subseteq X'\subseteq\alpha$ and $Y\subseteq Y'\subseteq\beta\backslash\alpha$. Let us consider~$\widetilde Y'\subseteq\alpha$ and $f:X'\cup Y'\to X'\cup\widetilde Y'$ as provided by~(iii). Since~$X'\cup Y'$ is closed, the implications in~(iii) upgrade to equivalences, as we have seen above. Thus~$f$ is an $\mathcal L_D$-isomorphism, and so is its restriction $f':X\cup Y\to\rng(f')$. As~$f$ fixes~$X'$, its restriction~$f'$ fixes $X\subseteq X'$ and has range $X\cup\widetilde Y$ for some $\widetilde Y\subseteq\alpha$. In view of Definition~\ref{def:leq_1^D} this yields~$\alpha\leq_1^D\beta$, as asserted by~(i).
\end{proof}

\section{From $\Sigma_1$-elementarity to Bachmann-Howard fixed points}\label{sect:elementarity-to-Bachmann-Howard}

Recall that a function $\vartheta:D(\alpha)\to\alpha$ is a Bachmann-Howard collapse of a given dilator~$D$ if conditions~(a) and~(b) from Theorem~\ref{thm:from-freund-equivalence} are satisfied. In this section we show how to transform~$D$ into a normal dilator~$\Sigma D$, such that a Bachmann-Howard collapse of~$D$ can be constructed if we have~$\alpha\leq_1^{\Sigma D}\Sigma D(\alpha+1)$.

The restriction to normal dilators is important for our approach to patterns of resemblance, because normality is required for Lemma~\ref{lem:normal-representation-independent}. On the other hand, the notion of Bachmann-Howard collapse is most interesting for dilators that are not normal, as the following remark shows. This explains why the aforementioned transformation of~$D$ into~$\Sigma D$ is necessary.

\begin{remark}
It is rather straightforward to construct a Bachmann-Howard fixed point of a normal dilator~$D$. In the previous section we have seen that~$\alpha\mapsto D(\alpha)$ is a normal function in the usual sense. We may thus pick a limit ordinal~\mbox{$\lambda=D(\lambda)$}. Let us define $\vartheta:D(\lambda)\to\lambda$ by setting~$\vartheta(\gamma):=D(\gamma+1)<D(\lambda)=\lambda$ for~\mbox{$\gamma<\lambda=D(\lambda)$}. Condition~(a) from Theorem~\ref{thm:from-freund-equivalence} is immediate since~$\vartheta$ is fully order preserving. Invoking equation~(\ref{eq:normality-extended}) and the discussion that follows it, we can also conclude that $\gamma<D(\gamma+1)=\mu_\lambda(\gamma+1)$ entails $\supp_\lambda(\gamma)\subseteq\gamma+1\leq\vartheta(\gamma)$, as required for~(b). To explain the significance of this observation, we recall that the existence of Bachmann-Howard fixed points for arbitrary dilators is equivalent to \mbox{$\Pi^1_1$-comprehension} (by Theorem~\ref{thm:from-freund-equivalence} above, proved in~\cite{freund-thesis,freund-equivalence,freund-categorical,freund-computable}). On the other hand, the fact that each normal function has a limit fixed point follows from (a~small amount of) $\Pi^1_1$-transfinite induction (by~\cite{freund-rathjen_derivatives,freund-ordinal-exponentiation,freund-single-fixed-point}). Since the latter is much weaker than $\Pi^1_1$-comprehension, this shows that the strength of Bachmann-Howard fixed points can only be exhausted by dilators that are not normal.
\end{remark}

Our dilators are defined as functors on natural numbers (cf.~Definition~\ref{def:pre-dilator}), even though they can be extended to arbitrary ordinals (via Definitions~\ref{def:dil-extend} and~\ref{def:dil-ordinals}). With this in mind, we define the normal variant of a given dilator as follows.

\begin{definition}\label{def:Sigma-D}
Let $D:\nat\to\ord$ be a pre-dilator. For $n\in\nat$ we define
\begin{equation*}
\Sigma D(n):=\Sigma_{k<n}1+D(k)=(1+D(0))+\dots+(1+D(n-1)),
\end{equation*}
where the right side refers to the usual operations of ordinal arithmetic. Each $\alpha<\Sigma D(n)$ can be uniquely written as $\alpha=\Sigma D(k)+\beta$ with $k<n$ and $\beta<1+D(k)$. Given a morphism $f:m\to n$ of~$\nat$, we define $\Sigma D(f):\Sigma D(m)\to\Sigma D(n)$ by
\begin{equation*}
\Sigma D(f)(\alpha):=\begin{cases}
\Sigma D(f(k)) & \text{if $\alpha=\Sigma D(k)$ with $k<m$},\\
\Sigma D(f(k))+1+D(f\!\restriction\!k)(\beta)\hspace*{-1ex} & \text{if $\alpha=\Sigma D(k)+1+\beta$ with $\beta<D(k)$},
\end{cases}
\end{equation*}
where $f\!\restriction\!k:k\to f(k)$ is the restriction of~$f$. Let us write $\supp^D:D\Rightarrow[\cdot]^{<\omega}$ for the natural transformation that comes with the pre-dilator~$D$. We put
\begin{equation*}
\supp^{\Sigma D}_n(\alpha):=\begin{cases}
\{k\} & \text{if $\alpha=\Sigma D(k)$ with $k<n$},\\
\{k\}\cup\supp^D_k(\beta) & \text{if $\alpha=\Sigma D(k)+1+\beta$ with $\beta<D(k)$},
\end{cases}
\end{equation*}
to define a family of functions $\supp^{\Sigma D}_n:\Sigma D(n)\to[n]^{<\omega}$. Finally, we construct functions $\mu_n:n\to\Sigma D(n)$ by setting $\mu_n(k):=\Sigma D(k)$ for $k<n$.
\end{definition}

It is straightforward to verify that $\Sigma D$ is a normal pre-dilator (cf.~the proof of \cite[Proposition~3.7]{freund-rathjen_derivatives}). To decide whether it is a dilator, we need to consider the orders~$\overline{\Sigma D}(\alpha)$ for infinite ordinals~$\alpha$. This is done as part of the following result, which will be fundamental for our construction of Bachmann-Howard fixed points.

\begin{proposition}\label{prop:D-into-SigmaD}
Assume that $D$ is a dilator. Then $\Sigma D$ is a normal dilator. For each ordinal~$\alpha$, we have an order isomorphism
\begin{equation*}
\xi_\alpha:D(\alpha)\to\{\delta\in\ord\,|\,\Sigma D(\alpha)<\delta<\Sigma D(\alpha+1)\},
\end{equation*}
such that $\supp^{\Sigma D}_{\alpha+1}(\xi_\alpha(\gamma))=\{\alpha\}\cup\supp^D_\alpha(\gamma)$ holds for all~$\gamma<D(\alpha)$.
\end{proposition}
\begin{proof}
We will first construct functions $\overline\xi_\alpha:\overline D(\alpha)\to\overline{\Sigma D}(\alpha+1)$ with analogous properties. The point is that these can be defined even when~$\overline{\Sigma D}(\alpha+1)$ is not well founded, while~$\Sigma D(\alpha+1)$ is undefined in that case (cf.~Definitions~\ref{def:dil-extend} and~\ref{def:dil-ordinals}). Using properties of~$\overline\xi_\alpha$, we will be able to confirm that $\overline{\Sigma D}(\alpha+1)$ is well founded after all. Based on this fact, it will be easy to see that $\overline\xi_\alpha$ induces~$\xi_\alpha$ as in the proposition. To describe the construction of~$\overline\xi_\alpha$, we recall that elements of $\overline D(\alpha)$ have the form $(\sigma;\alpha_0,\dots,\alpha_{n-1};\alpha)_D$ with $\alpha_0<\dots<\alpha_{n-1}<\alpha$ and $(\sigma,n)\in\tr(D)$. The last condition is the conjunction of~$\sigma\in D(n)$ and $\supp^D_n(\sigma)=n$. We can conclude $\Sigma D(n)+1+\sigma\in\Sigma D(n+1)$ and
\begin{equation*}
\supp^{\Sigma D}_{n+1}(\Sigma D(n)+1+\sigma)=\{n\}\cup\supp^D_n(\sigma)=\{0,\dots,n\}=n+1,
\end{equation*}
which amounts to $(\Sigma D(n)+1+\sigma,n+1)\in\tr(\Sigma D)$. This allows us to set
\begin{equation*}
\overline\xi_\alpha\left((\sigma;\alpha_0,\dots,\alpha_{n-1};\alpha)_D\right):=(\Sigma D(n)+1+\sigma;\alpha_0,\dots,\alpha_{n-1},\alpha;\alpha+1)_{\Sigma D}.
\end{equation*}
It is straightforward but somewhat tedious to show that $\overline\xi_\alpha$ is strictly increasing: Consider an inequality
\begin{equation*}
(\sigma;\alpha_0,\dots,\alpha_{m-1};\alpha)_D<_{\overline D(\alpha)}(\tau;\beta_0,\dots,\beta_{n-1};\alpha)_D.
\end{equation*}
In view of Definition~\ref{def:dil-extend} and the paragraph that precedes it, we have
\begin{equation*}
D(|\iota_a^{a\cup b}|)(\sigma)<_{D(|a\cup b|)}D(|\iota_b^{a\cup b}|)(\tau)
\end{equation*}
with $a=\{\alpha_0,\dots,\alpha_{m-1}\}$ and $b=\{\beta_0,\dots,\beta_{n-1}\}$. Set $c:=a\cup\{\alpha\}$ and $d:=b\cup\{\alpha\}$. We then have $|c|=|a|+1$, and the strictly increasing enumeration $\en_c:|c|\to c$ can be characterized by
\begin{equation*}
\en_c(i)=\begin{cases}
\en_a(i) & \text{if $i<|a|$},\\
\alpha & \text{if $i=|a|$}.
\end{cases}
\end{equation*}
An analogous description is available for~$\en_{c\cup d}:|c\cup d|=|a\cup b|+1\to c\cup d$. Now consider the function~$h:|c|\to|c\cup d|$ with
\begin{equation*}
h(i):=\begin{cases}
|\iota_a^{a\cup b}|(i) & \text{if $i<|a|$},\\
|a\cup b| & \text{if $i=|a|$}.
\end{cases}
\end{equation*}
For $i<|a|$ we can compute
\begin{equation*}
\en_{c\cup d}\circ h(i)=\en_{c\cup d}\circ\underbrace{|\iota_a^{a\cup b}|(i)}_{<|a\cup b|}=\en_{a\cup b}\circ|\iota_a^{a\cup b}|(i)=\iota_a^{a\cup b}\circ\en_a(i)=\iota_c^{c\cup d}\circ\en_c(i).
\end{equation*}
The equation between the outermost expressions does also hold for~$i=|a|$ (where both are equal to~$\alpha$). It follows that $h$ is equal to $|\iota_c^{c\cup d}|$, which is uniquely characterized by the equation that we have established. In other words, we have
\begin{equation*}
|\iota_c^{c\cup d}|(|a|)=|a\cup b|\quad\text{and}\quad|\iota_c^{c\cup d}|\!\restriction\!|a|=|\iota_a^{a\cup b}|:|a|\to|a\cup b|.
\end{equation*}
The analogous facts hold for $|\iota_c^{c\cup d}|$ (with $|a|$ and $\iota_a^{a\cup b}$ replaced by~$|b|$ and $\iota_b^{a\cup b}$). By Definition~\ref{def:Sigma-D} (with $m=|a|$ and $n=|b|$) and the inequality from above we get
\begin{multline*}
\Sigma D(|\iota_c^{c\cup d}|)(\Sigma D(m)+1+\sigma)=\Sigma D(|a\cup b|)+1+D(|\iota_a^{a\cup b}|)(\sigma)<\\
<\Sigma D(|a\cup b|)+1+D(|\iota_b^{a\cup b}|)(\tau)=\Sigma D(|\iota_d^{c\cup d}|)(\Sigma D(n)+1+\tau).
\end{multline*}
In view of Definition~\ref{def:dil-extend} this yields
\begin{equation*}
\overline\xi_\alpha\left((\sigma;\alpha_0,\dots,\alpha_{m-1};\alpha)_D\right)<_{\overline{\Sigma D}(\alpha+1)}\overline\xi_\alpha\left((\tau;\beta_0,\dots,\beta_{n-1};\alpha)_D\right),
\end{equation*}
as desired. In particular, $\overline\xi_\alpha$ is injective. According to Definition~\ref{def:normal-extend}, the natural transformation~$\mu:I\Rightarrow\Sigma D$ induces functions~$\overline\mu_\gamma:\gamma\to\overline{\Sigma D}(\gamma)$. We now show
\begin{equation*}
\rng(\xi_\alpha)=\{\rho\in\overline{\Sigma D}(\alpha+1)\,|\,\overline\mu_{\alpha+1}(\alpha)<_{\overline{\Sigma D}(\alpha+1)}\rho\}.
\end{equation*}
For~$\subseteq$ we use equation~(\ref{eq:normality-extended-overline}) to get~$\overline\mu_{\alpha+1}(\alpha)\leq_{\overline{\Sigma D}(\alpha+1)}\xi_\alpha(\pi)$ for any~$\pi\in\overline D(\alpha)$ (the point is that values of~$\xi_\alpha$ have the form $(\,\cdot\,;\,\cdot\,,\ldots,\,\cdot\,,\alpha;\alpha+1)_{\Sigma D}$ with component~$\alpha$). It remains to show $\overline\mu_{\alpha+1}(\alpha)\neq\xi_\alpha(\pi)$. This holds because we have
\begin{equation*}
\overline\mu_{\alpha+1}(\alpha)=(\mu_1(0);\alpha;\alpha+1)_{\Sigma D}=(\Sigma D(0);\alpha;\alpha+1)_{\Sigma D},
\end{equation*}
while values of~$\xi_\alpha$ have first entries of the form~$\Sigma D(n)+1+\sigma\neq\Sigma D(0)$. To establish the implication~$\supseteq$ of the equality above, we consider an arbitrary element
\begin{equation*}
\rho=(\rho_0;\alpha_0,\dots,\alpha_{k-1};\alpha+1)_{\Sigma D}\in\overline{\Sigma D}(\alpha+1).
\end{equation*}
Assuming $\overline\mu_{\alpha+1}(\alpha)<_{\overline{\Sigma D}(\alpha+1)}\rho$, we can once again invoke equation~(\ref{eq:normality-extended-overline}) to get $k>0$ and $\alpha_{k-1}=\alpha$. Let us observe that $\rho_0$ cannot be of the form $\Sigma D(n)$, since this would yield $(\Sigma D(n),k)\in\tr(\Sigma D)$, hence $k=\supp^{\Sigma D}_k(\Sigma D(n))=\{n\}$, then $n=0$ and $k=1$, and finally $\rho=(\Sigma D(0);\alpha;\alpha+1)_D=\overline\mu_{\alpha+1}(\alpha)$. This means that we must have $\rho_0=\Sigma D(n)+1+\rho_1$ for some $n<k$ and $\rho_1\in D(n)$. We again get
\begin{equation*}
k=\supp^{\Sigma D}_k(\Sigma D(n)+1+\rho_1)=\{n\}\cup\supp^D_n(\rho_1),
\end{equation*}
which entails $k=n+1$ and $\supp^D_n(\rho_1)=n$, so that we have $(\rho_1,n)\in\tr(D)$. Due to the latter, we may consider
\begin{equation*}
\rho':=(\rho_1;\alpha_0,\dots,\alpha_{n-1};\alpha)_D\in\overline D(\alpha).
\end{equation*}
By construction we have $\rho=\xi_\alpha(\rho')\in\rng(\xi_\alpha)$, as desired. Assuming that $D$ is a dilator, we can now show that the same holds for~$\Sigma D$. Towards a contradiction we assume that $f:\mathbb N\to\overline{\Sigma D}(\beta)$ is strictly decreasing, for some $\beta\in\ord$. Note that we must have~$\beta>0$ (as $\overline{\Sigma D}(0)=\Sigma D(0)=0$), and that $\overline\mu_\beta(0)=(\Sigma D(0);0;\beta)_{\Sigma D}$ is the smallest element of~$\overline{\Sigma D(\beta)}$ (using Definition~\ref{def:dil-extend}). We may thus consider the minimal~$\alpha<\beta$ such that $\overline\mu_\beta(\alpha)\leq_{\overline{\Sigma D}(\beta)}f(n)$ holds for all~$n\in\mathbb N$. Since~$f$ is strictly increasing, we must have
\begin{equation*}
\overline\mu_\beta(\alpha)<_{\overline{\Sigma D}(\beta)}f(n)<_{\overline{\Sigma D}(\beta)}\overline\mu_\beta(\alpha+1)
\end{equation*}
for all sufficiently large~$n$ (where the second inequality is dropped in case~$\beta=\alpha+1$). Possibly after shifting~$f$, we may assume that these inequalities hold for all~$n\in\mathbb N$. If $\iota_{\alpha+1}^\beta:\alpha+1\hookrightarrow\beta$ is the inclusion, $\overline{\Sigma D}(\iota_{\alpha+1}^\beta):\overline{\Sigma D}(\alpha+1)\to\overline{\Sigma D}(\beta)$ has range
\begin{equation*}
\rng(\overline{\Sigma D}(\iota_{\alpha+1}^\beta))=\{\rho\in\overline{\Sigma D}(\beta)\,|\,\rho<_{\overline{\Sigma D}(\beta)}\overline\mu_\beta(\alpha+1)\},
\end{equation*}
as in the discussion that follows Definition~\ref{def:normal-extend}. We thus get a strictly decreasing function $g:\mathbb N\to\overline{\Sigma D}(\alpha+1)$ with $\overline{\Sigma D}(\iota_{\alpha+1}^\beta)\circ g=f$. In view of
\begin{multline*}
\overline{\Sigma D}(\iota_{\alpha+1}^\beta)(\overline\mu_{\alpha+1}(\alpha))=\overline{\Sigma D}(\iota_{\alpha+1}^\beta)((\Sigma D(0);\alpha;\alpha+1)_{\Sigma D})=\\
=(\Sigma D(0);\iota_{\alpha+1}^\beta(\alpha);\beta)_{\Sigma D}=(\Sigma D(0);\alpha;\beta)_{\Sigma D}=\overline\mu_\beta(\alpha)
\end{multline*}
we have $\overline\mu_{\alpha+1}(\alpha)<_{\overline{\Sigma D}(\alpha+1)}g(n)$ for all~$n\in\mathbb N$. Now the above allows us to specify a strictly increasing function $h:\mathbb N\to\overline D(\alpha)$ by stipulating $\xi_\alpha\circ h=g$. This contradicts the assumption that~$D$ is a dilator. Once we know that $D$ and $\Sigma D$ are dilators, we can consider the isomorphisms $\eta^D_\gamma:D(\gamma)\to\overline D(\gamma)$ and $\eta^{\Sigma D}_\gamma:\Sigma D(\gamma)\to\overline{\Sigma D}(\gamma)$ from Definition~\ref{def:dil-ordinals}, where $D(\gamma)$ and $\Sigma D(\gamma)$ are ordinals. Let $\xi_\alpha:D(\alpha)\to\Sigma D(\alpha+1)$ be determined by $\eta^{\Sigma D}_{\alpha+1}\circ\xi_\alpha=\overline\xi_\alpha\circ\eta^D_\alpha$. The claim that we have
\begin{equation*}
\rng(\xi_\alpha)=\{\delta\in\ord\,|\,\Sigma D(\alpha)<\delta<\Sigma D(\alpha+1)\}
\end{equation*}
can be derived from the corresponding result about $\overline\xi_\alpha$, using $\eta^{\Sigma D}_{\alpha+1}\circ\mu_{\alpha+1}=\overline\mu_{\alpha+1}$ (by Definition~\ref{def:normal-extend}) and $\mu_{\alpha+1}(\alpha)=\Sigma D(\alpha)$ (by the discussion after that definition). Finally, for~$\gamma<D(\alpha)$ with $\eta^D_\alpha(\gamma)=(\sigma;\alpha_0,\dots,\alpha_{n-1};\alpha)_D$ we have
\begin{equation*}
\eta^{\Sigma D}_\alpha(\xi_\alpha(\gamma))=\overline\xi_\alpha(\eta^D_\alpha(\gamma))=(\Sigma D(n)+1+\sigma;\alpha_0,\dots,\alpha_{n-1},\alpha;\alpha+1)_{\Sigma D}.
\end{equation*}
By Definition~\ref{def:dil-ordinals} we obtain
\begin{equation*}
\supp^{\Sigma D}_\alpha(\xi_\alpha(\gamma))=\{\alpha_0,\dots,\alpha_{n-1}\}\cup\{\alpha\}=\supp^D_\alpha(\gamma)\cup\{\alpha\},
\end{equation*}
as claimed in the proposition.
\end{proof}

Given a normal dilator~$E$, Definition~\ref{def:representation} and Proposition~\ref{prop:representations-exist-unique} provide unique representations $\gamma\simeq(\sigma;\gamma_0,\dots,\gamma_{n-1})_E$ of all ordinals. We point out that $\gamma\geq E(0)$ is equivalent to $n>0$, by the same proposition. In the present section we only consider $E=\Sigma D$, where $\Sigma D(0)=0\leq\gamma$ is automatic. The general case will be needed in the next section.

\begin{definition}\label{def:fund-seq}
Let $E$ be a normal dilator. Given an ordinal~$\gamma\geq E(0)$, we define
\begin{equation*}
\gamma^*:=\sup\{\gamma_i+1\,|\,i<n\}\quad\text{for }\gamma\simeq(\sigma;\gamma_0,\dots,\gamma_n)_E.
\end{equation*}
If we have $\delta\geq\gamma^*$, then we define $\gamma[\delta]\in\ord$ by stipulating
\begin{equation*}
\gamma[\delta]\simeq(\sigma;\gamma_0,\dots,\gamma_{n-1},\delta)_E,
\end{equation*}
where we still assume~$\gamma\simeq(\sigma;\gamma_0,\dots,\gamma_n)_E$.
\end{definition}

Parts~(a) and~(d) of the following are used in our construction of a Bachmann-Howard collapse (cf.~Theorem~\ref{thm:patterns-to-collapse}). The other parts are needed for the next section.

\begin{lemma}\label{lem:fund-basic}
The following holds for any normal dilator~$E$ and all~$\beta,\gamma\geq E(0)$:
\begin{enumerate}[label=(\alph*)]
\item We have $E(\delta)\leq\gamma[\delta]<E(\delta+1)$ for any ordinal~$\delta\geq\gamma^*$.
\item From $E(\delta)\leq\gamma<E(\delta+1)$ we can infer $\delta\geq\gamma^*$ and $\gamma[\delta]=\gamma$.
\item For $\delta,\rho\geq\gamma^*$ we have $\gamma[\delta]^*=\gamma^*$ and $\gamma[\delta][\rho]=\gamma[\rho]$.
\item If we have $\delta\geq\max\{\beta^*,\gamma^*\}$ and $E(\rho)\leq\beta,\gamma<E(\rho+1)$ for some~$\rho$, then
\begin{equation*}
\beta<\gamma\quad\Leftrightarrow\quad\beta[\delta]<\gamma[\delta].
\end{equation*}
\item If we have $E(\rho)\leq\gamma<\gamma+1<E(\rho+1)$ with $\rho\geq\delta\geq\max\{\gamma^*,(\gamma+1)^*\}$, then we have $(\gamma+1)[\delta]=\gamma[\delta]+1$.
\item We have $E(\rho)^*=0$ and $E(\rho)[\delta]=E(\delta)$, for arbitrary ordinals~$\rho$ and $\delta$.
\item Assume that we have $E(\rho)<\gamma<E(\rho+1)$ with $\rho\geq\delta\geq\gamma^*$. If $\gamma$ and $\delta$ are limit ordinals, then so is~$\gamma[\delta]$.
\end{enumerate}
\end{lemma}
\begin{proof}
Parts~(a) to~(c) are easy consequences of Proposition~\ref{prop:representations-exist-unique} and Definition~\ref{def:fund-seq}.

(d) Choose an additively closed ordinal number~$\pi>\max\{\delta,\rho\}$, consider the isomorphism $\eta_\pi:E(\pi)\to\overline E(\pi)$ from Definition~\ref{def:dil-ordinals}, and write
\begin{equation*}
\eta_\pi(\beta)=(\sigma;\beta_0,\dots,\beta_m;\pi)_E\quad\text{and}\quad\eta_\pi(\gamma)=(\tau;\gamma_0,\dots,\gamma_n;\pi)_E.
\end{equation*}
In view of Definition~\ref{def:representation} and Proposition~\ref{prop:representations-exist-unique} we have $\beta_m=\rho=\gamma_n$. Also note
\begin{equation*}
\eta_\pi(\beta[\delta])=(\sigma;\beta_0,\dots,\beta_{m-1},\delta;\pi)_E\quad\text{and}\quad\eta_\pi(\gamma[\delta])=(\tau;\gamma_0,\dots,\gamma_{n-1},\delta;\pi)_E.
\end{equation*}
Assume $\delta\leq\rho$ (the case $\rho<\delta$ being analogous), and define $f:\pi\to\pi$ by
\begin{equation*}
f(\alpha):=\begin{cases}
\alpha & \text{if $\alpha<\delta$},\\
\rho+\beta & \text{if $\alpha=\delta+\beta<\pi$}.
\end{cases}
\end{equation*}
In view of Definition~\ref{def:dil-extend} we get
\begin{equation*}
\overline E(f)\left(\eta_\pi(\beta[\delta])\right)=\eta_\pi(\beta)\quad\text{and}\quad\overline E(f)\left(\eta_\pi(\gamma[\delta])\right)=\eta_\pi(\gamma).
\end{equation*}
Now the claim follows since $\overline E(f)$ is an order embedding (cf.~\cite[Lemma~2.2]{freund-computable}).

(e) Part~(d) yields $\gamma[\delta]+1\leq(\gamma+1)[\delta]$. We thus have $\eta:=\gamma[\delta]+1<E(\delta+1)$ and hence $\eta^*\leq\delta\leq\rho$, by part~(b). By the latter and part~(c) we get $\eta=\eta[\delta]=\eta[\rho][\delta]$. We now derive a contradiction from $\gamma[\delta]+1<(\gamma+1)[\delta]$. The latter yields
\begin{equation*}
\gamma[\delta]<\eta[\rho][\delta]<(\gamma+1)[\delta].
\end{equation*}
As~(a) provides $E(\rho)\leq\eta[\rho]<E(\rho+1)$, we can invoke~(d) to get $\gamma<\eta[\rho]<\gamma+1$, which is indeed impossible.

(f) Recall that~$E$ comes with a natural transformation $\mu:I\Rightarrow E$ that induces a function $\mu_{\rho+1}:\rho+1\to E(\rho+1)$ with $\mu_{\rho+1}(\rho)=E(\rho)$, by Definition~\ref{def:normal-extend} and the discussion that follows it. We compute
\begin{equation*}
\eta_{\rho+1}(E(\rho))=\eta_{\rho+1}(\mu_{\rho+1}(\rho))=\overline\mu_{\rho+1}(\rho)=(\mu_1(0);\rho;\rho+1)_E.
\end{equation*}
This yields $E(\rho)\simeq(\mu_1(0);\rho)_E$, which makes the claims obvious.

(g) By the previous parts we get $E(\delta)=E(\rho)[\delta]<\gamma[\delta]<E(\delta+1)$. Aiming at a contradiction, assume that we have $\gamma[\delta]=\alpha+1$ with~$E(\delta)\leq\alpha<E(\delta+1)$. The latter yields $\alpha^*\leq\delta$, then $\alpha[\rho][\delta]=\alpha[\delta]=\alpha<\gamma[\delta]$, and finally $\alpha[\rho]<\gamma$. Now the assumption that~$\gamma$ is a limit allows us to pick an ordinal~$\beta$ with $\alpha[\rho]<\beta<\gamma$. Since~$\delta$ is a limit, we have $\xi:=\max\{\alpha^*,\gamma^*\}<\delta$. Write $\beta\simeq(\sigma;\beta_0,\dots,\beta_{n-1},\rho)_E$, and let $i\leq n$ be such that $\beta_{i-1}<\xi\leq\beta_i$ (where $\beta_{-1}:=-1<\xi$ and $\beta_n:=\rho>\xi$). Now set $\zeta:=\xi+n-i<\delta$, and define a strictly increasing $f:\zeta+1\to\rho+1$ by
\begin{equation*}
f(\alpha):=\begin{cases}
\alpha & \text{if $\alpha<\xi$},\\
\beta_{i+j} & \text{if $\alpha=\xi+j$ with $j<n-i$},\\
\rho & \text{if $\alpha=\xi+n-i=\zeta$}.
\end{cases}
\end{equation*}
For $\eta_{\zeta+1}:E(\zeta+1)\to\overline E(\zeta+1)$ as in Definition~\ref{def:dil-ordinals}, we define $\beta'<E(\zeta+1)$ by
\begin{equation*}
\eta_{\zeta+1}(\beta')=(\sigma;\beta_0,\dots,\beta_{i-1},\xi+i-i,\dots,\xi+n-i;\zeta+1)_E.
\end{equation*}
It is straightforward to verify $E(f)(\beta')=\beta$ (use $\eta_{\rho+1}\circ E(f)=\overline E(f)\circ\eta_{\zeta+1}$). In view of~$\alpha^*,\gamma^*\leq\xi$ we also get $E(f)(\alpha[\zeta])=\alpha[\rho]$ and $E(f)(\gamma[\zeta])=\gamma$. Since $E(f)$ is an order embedding, we can conclude $\alpha[\zeta]<\beta'<\gamma[\zeta]$ and then
\begin{equation*}
\alpha=\alpha[\delta]=\alpha[\zeta][\delta]<\beta'[\delta]<\gamma[\zeta][\delta]=\gamma[\delta],
\end{equation*}
which is the desired contradiction with~$\gamma[\delta]=\alpha+1$.
\end{proof}

Let us also relate the new notation to the functions $\xi_\alpha:D(\alpha)\to\Sigma D(\alpha+1)$ from Proposition~\ref{prop:D-into-SigmaD}. We note that $\xi_\alpha(\gamma)^*$ is computed with respect to~$E=\Sigma D$.

\begin{lemma}\label{lem:supp-star}
For each $\gamma<D(\alpha)$ we have $\xi_\alpha(\gamma)^*=\min\{\delta\in\ord\,|\,\supp^D_\alpha(\gamma)\subseteq\delta\}$.
\end{lemma}
\begin{proof}
In view of $\Sigma D(\alpha)<\xi_\alpha(\gamma)<\Sigma D(\alpha+1)$ we get
\begin{equation*}
\eta^{\Sigma D}_{\alpha+1}(\xi_\alpha(\gamma))=(\sigma;\gamma_0,\dots,\gamma_n;\alpha+1)_{\Sigma D}
\end{equation*}
with $\gamma_0<\ldots<\gamma_n=\alpha$ (cf.~Proposition~\ref{prop:representations-exist-unique}). According to Definition~\ref{def:dil-ordinals} we have
\begin{equation*}
\supp^{\Sigma D}_{\alpha+1}(\xi_\alpha(\gamma))=\{\gamma_0,\dots,\gamma_n\}=\{\gamma_0,\dots,\gamma_{n-1}\}\cup\{\alpha\}.
\end{equation*}
By Proposition~\ref{prop:D-into-SigmaD} we get $\supp^D_\alpha(\gamma)=\{\gamma_0,\dots,\gamma_{n-1}\}$, so that the claim follows from the definition of~$\xi_\alpha(\gamma)^*$.
\end{proof}

Finally, we establish the promised connection between patterns of resemblance and Bachmann-Howard fixed points. Recall that an ordinal~$\alpha$ is such a fixed point if there is a function~$\vartheta:D(\alpha)\to\alpha$ as in statement~(i) from Theorem~\ref{thm:from-freund-equivalence}.

\begin{theorem}\label{thm:patterns-to-collapse}
Let~$D$ be a dilator. If we have $\alpha\leq_1^{\Sigma D}\Sigma D(\alpha+1)$, then $\alpha$ is a Bachmann-Howard fixed point of~$D$.
\end{theorem}
\begin{proof}
In order to construct a Bachmann-Howard collapse~$\vartheta:D(\alpha)\to\alpha$, we consider an arbitrary ordinal~$\gamma<D(\alpha)$. Let $\xi_\alpha:D(\alpha)\to\Sigma D(\alpha+1)$ be the function from Proposition~\ref{prop:D-into-SigmaD}. Due to~$\Sigma D(\alpha)<\xi_\alpha(\gamma)<\Sigma D(\alpha+1)$, Proposition~\ref{prop:representations-exist-unique} yields a representation of the form
\begin{equation*}
\xi_\alpha(\gamma)\simeq(\sigma;\gamma_0,\dots,\gamma_n)_{\Sigma D}\quad\text{with }\gamma_0<\ldots<\gamma_n=\alpha.
\end{equation*}
From $\alpha\leq_1^{\Sigma D}\Sigma D(\alpha+1)$ and $\alpha\leq\Sigma D(\alpha)<\xi_\alpha(\gamma)<\Sigma D(\alpha+1)$ we get~$\alpha\leq_1^{\Sigma D}\xi_\alpha(\gamma)$, as a glance at Definition~\ref{def:leq_1^D} reveals. Now we apply the latter to $\alpha\leq_1^{\Sigma D}\Sigma D(\alpha+1)$ and the sets $X=\{\gamma_0,\dots,\gamma_{n-1}\}\subseteq\alpha$ and $Y=\{\alpha,\xi_\alpha(\gamma)\}\in\Sigma D(\alpha+1)\backslash\alpha$. This yields a set $\widetilde Y=\{\delta,\eta\}\subseteq\alpha$ such that we have
\begin{equation*}
\eta\simeq(\sigma;\gamma_0,\dots,\gamma_{n-1},\delta)_{\Sigma D}\quad\text{and}\quad\delta\leq_1^{\Sigma D}\eta.
\end{equation*}
The first conjunct does, in particular, entail $\{\gamma_0,\dots,\gamma_{n-1}\}\subseteq\delta$. In the notation from Definition~\ref{def:fund-seq} we have $\delta\geq\xi_\alpha(\gamma)^*$ and $\eta=\xi_\alpha(\gamma)[\delta]$. We can thus set
\begin{equation*}
\vartheta(\gamma):=\min\{\delta<\alpha\,|\,\delta\geq\xi_\alpha(\gamma)^*\text{ and }\delta\leq_1^{\Sigma D}\xi_\alpha(\gamma)[\delta]\}.
\end{equation*}
Concerning the implementation in our base theory~$\textsf{ATR}_0^{\textsf{set}}$, we point out that $\xi_\alpha(\gamma)^*$ and $\xi_\alpha(\gamma)[\delta]$ can be computed with parameter $\eta^{\Sigma D}_{\alpha+1}:\Sigma D(\alpha+1)\to\overline{\Sigma D}(\alpha+1)$ (which is needed to determine representations). It remains to verify conditions~(a) and~(b) from Theorem~\ref{thm:from-freund-equivalence}. By Lemma~\ref{lem:supp-star} and the definition of~$\vartheta$ we get
\begin{equation*}
\supp^D_\alpha(\gamma)\subseteq\xi_\alpha(\gamma)^*\leq\vartheta(\gamma),
\end{equation*}
which is condition~(b). To establish condition~(a), consider ordinals $\gamma<\gamma'<D(\alpha)$ with~$\supp^D_\alpha(\gamma)\subseteq\vartheta(\gamma')$. The latter yields $\vartheta(\gamma')\geq\xi_\alpha(\gamma)^*$, again by Lemma~\ref{lem:supp-star}. Due to the definition of~$\vartheta$, we also have $\vartheta(\gamma')\geq\xi_\alpha(\gamma')^*$. Using properties of $\xi_\alpha$, we can derive $\Sigma D(\alpha)<\xi_\alpha(\gamma)<\xi_\alpha(\gamma')<\Sigma D(\alpha+1)$ and then
\begin{equation*}
\vartheta(\gamma')\leq\Sigma D(\vartheta(\gamma'))\leq\xi_\alpha(\gamma)[\vartheta(\gamma')]<\xi_\alpha(\gamma')[\vartheta(\gamma')],
\end{equation*}
by Lemma~\ref{lem:fund-basic}. Also by the definition of~$\vartheta$, we get $\vartheta(\gamma')\leq_1^{\Sigma D}\xi_\alpha(\gamma')[\vartheta(\gamma')]$ and then $\vartheta(\gamma')\leq_1^{\Sigma D}\xi_\alpha(\gamma)[\vartheta(\gamma')]$. Let us write the representation of $\xi_\alpha(\gamma)$ as above, so that we have
\begin{equation*}
\xi_\alpha(\gamma)[\vartheta(\gamma')]\simeq(\sigma;\gamma_0,\dots,\gamma_{n-1},\vartheta(\gamma'))_{\Sigma D}.
\end{equation*}
We now apply Definition~\ref{def:leq_1^D} to the relation $\vartheta(\gamma')\leq_1^{\Sigma D}\xi_\alpha(\gamma')[\vartheta(\gamma')]$ and the sets $X=\{\gamma_0,\dots,\gamma_{n-1}\}$ and $Y=\{\vartheta(\gamma'),\xi_\alpha(\gamma)[\vartheta(\gamma')]\}$. This yields a $\widetilde Y=\{\delta,\eta\}\subseteq\vartheta(\gamma')$ with $\eta\simeq(\sigma;\gamma_0,\dots,\gamma_{n-1},\delta)_{\Sigma D}$ and $\delta\leq_1^{\Sigma D}\eta$. Once again we have $\delta\geq\xi_\alpha(\gamma)^*$ as well as $\eta=\xi_\alpha(\gamma)[\delta]$. Minimality yields $\vartheta(\gamma)\leq\delta<\vartheta(\gamma')$, as condition~(a) demands.
\end{proof}

\section{From admissible sets to $\Sigma_1$-elementarity}\label{sect:admissible-to-elementarity}

In this section we show that $\Omega\leq_1^D D(\Omega+1)$ holds when $\Omega$ is the ordinal height of an admissible set that contains the normal dilator~$D$. Together with the result from the previous section (and the result of~\cite{freund-equivalence}), this will allow us to establish Theorem~\ref{thm:pattern-admissibles} from the introduction (see the proof at the end of this section).

Throughout the following we fix a normal dilator~$D$. The crucial idea is to consider the classes
\begin{equation*}
C_D(\gamma):=\{\delta\in\ord\,|\,\delta\geq\gamma^*\text{ and }\delta\leq_1^D\gamma[\delta]\},
\end{equation*}
for $\gamma\geq D(0)$ (using the notation from Definition~\ref{def:fund-seq}). For~$\Omega$ as in the previous paragraph, we will show that $C_D(\gamma)\cap\Omega$ is closed and unbounded (club) in~$\Omega$, by induction from~$\gamma=\Omega=D(\Omega)$ up to arbitrary~$\gamma<D(\Omega+1)$. The following yields the base case, as the fixed points of the normal function $\alpha\mapsto D(\alpha)$ do form a club.

\begin{lemma}\label{lem:clubs-basis}
If $\gamma=D(\rho)$ holds for some~$\rho$, we have $C_D(\gamma)=\{\delta\in\ord\,|\,D(\delta)=\delta\}$.
\end{lemma}
\begin{proof}
By Lemma~\ref{lem:fund-basic} we have $\gamma^*=0$ and $\gamma[\delta]=D(\delta)$. It thus remains to show that $D(\delta)=\delta$ is equivalent to $\delta\leq_1^D D(\delta)$. For the direction from left to right, it suffices to observe that $\leq_1^D$ is reflexive. To establish the other direction, we derive a contradiction from the assumption that we have $\delta\leq_1^D D(\delta)$ and $\delta<D(\delta)$. In view of the latter, Proposition~\ref{prop:representations-exist-unique} provides a representation
\begin{equation*}
\delta\simeq(\sigma;\delta_0,\dots,\delta_{n-1})_D\quad\text{with }\delta_0<\ldots<\delta_{n-1}<\delta.
\end{equation*}
Let us now apply Definition~\ref{def:leq_1^D} to $\delta\leq_1^D D(\delta)$ and the sets $X=\{\delta_0,\dots,\delta_{n-1}\}\subseteq\delta$ and $Y=\{\delta\}\subseteq D(\delta)\backslash\delta$ (invoking~$\delta<D(\delta)$ again). This yields a $\widetilde Y=\{\delta'\}\subseteq\delta$ with $\delta'\simeq(\sigma;\delta_0,\dots,\delta_{n-1})_D$, contradicting the uniqueness part of Proposition~\ref{prop:representations-exist-unique}.
\end{proof}

The following result (essentially an abstract version of~\cite[Lemma~3.11]{wilken-bachmann-howard}) is needed for the successor case. Let us recall that~$\delta>0$ is a limit point of a class $C\subseteq\ord$ if any $\beta<\delta$ admits a $\gamma\in C\cap\delta$ with $\beta<\gamma$. The assumption~$\delta\in C_D(\gamma)$ in the following result is in fact automatic, by Corollary~\ref{cor:club-closed} below.

\begin{proposition}\label{prop:club-successor}
Consider an ordinal~$\gamma\geq D(0)$. If $\delta$ is an element and a limit point of~$C_D(\gamma)$, then we have $\delta\leq_1^D\gamma[\delta]+1$.
\end{proposition}
If the assumption of Lemma~\ref{lem:fund-basic}(e) is satisfied, then we have $\gamma[\delta]+1=(\gamma+1)[\delta]$, so that the proposition yields $\delta\in C_D(\gamma+1)$.
\begin{proof}
In order to show $\delta\leq_1^D\gamma[\delta]+1$, we will apply the criterion from part~(c) of Theorem~\ref{prop:leq_1^D-Sigma_1-elem}. To this end, let us consider finite sets $X\subseteq\delta$ and $Y\subseteq(\gamma[\delta]+1)\backslash\delta$. If~$\delta$ is a limit of ordinals~$\beta$ with $\alpha\leq_1^D\beta$, then we also have $\alpha\leq_1^D\delta$, as a glance at Definition~\ref{def:leq_1^D} reveals. For each $\alpha\in X$ with $\alpha\not\leq_1^D\delta$, we may thus pick an~$\alpha'>\alpha$ with $\alpha\not\leq_1^D\alpha'<\delta$. Let us set
\begin{equation*}
X':=\{\alpha'\,|\,\alpha\in X\text{ and }\alpha\not\leq_1^D\delta\}\subseteq\delta.
\end{equation*}
From $\delta\in C_D(\gamma)$ we can derive $\delta\leq_1^D D(\delta)\leq\gamma[\delta]$ and then $D(\delta)=\delta$, as in the proof of Lemma~\ref{lem:clubs-basis}. In particular we have $\beta>D(0)$ for any ordinal~$\beta\in Y$. Due to $D(\delta)=\delta\leq\beta\leq\gamma[\delta]<D(\delta+1)$ we get $\beta^*\leq\delta$, by Lemma~\ref{lem:fund-basic}. In fact, we even obtain $\beta^*<\delta$, as $\beta^*$ is never a limit (cf.~Definition~\ref{def:fund-seq}). Using the assumption of the proposition, we now choose an ordinal~$\eta$ with
\begin{equation*}
X\cup X'\cup\{\beta^*\,|\,\beta\in Y\}\subseteq\eta\in C_D(\gamma)\cap\delta.
\end{equation*}
Put $\widetilde Y:=\{\beta[\eta]\,|\,\beta\in Y\}$, and define $f:X\cup Y\to X\cup\widetilde Y$ by
\begin{equation*}
f(\beta):=\begin{cases}
\beta & \text{if $\beta\in X$},\\
\beta[\eta] & \text{if $\beta\in Y$}.
\end{cases}
\end{equation*}
It remains to verify the conditions from part~(c) of Theorem~\ref{prop:leq_1^D-Sigma_1-elem}. First observe that the elements of~$\widetilde Y$ satisfy $\beta[\eta]<D(\eta+1)\leq D(\delta)=\delta$. We now show
\begin{equation*}
\alpha\leq\beta\quad\Leftrightarrow\quad f(\alpha)\leq f(\beta)
\end{equation*}
for $\alpha,\beta\in X\cup Y$. If we have $\alpha\in X$ and $\beta\in Y$, the left side holds and we have
\begin{equation*}
f(\alpha)=\alpha<\eta\leq D(\eta)\leq\beta[\eta]=f(\beta).
\end{equation*}
The case of~$\alpha\in Y$ and $\beta\in X$ is covered by essentially the same argument. It remains to consider $\alpha,\beta\in Y$. Here we have $D(\delta)=\delta\leq\alpha,\beta\leq\gamma[\delta]<D(\delta+1)$. From Lemma~\ref{lem:fund-basic} we learn that $\alpha<\beta$ is equivalent to~$\alpha[\eta]<\beta[\eta]$. The same equivalence holds when both occurrences of~$<$ are replace by~$\leq$, as we are concerned with a linear order. Next, we establish
\begin{equation*}
\alpha\leq_1^D\beta\quad\Leftrightarrow\quad f(\alpha)\leq_1^D f(\beta).
\end{equation*}
We may focus on the case of~$\alpha<\beta$, as we have seen that~$f$ is an order isomorphism. Let us first assume $\alpha\in X$ and $\beta\in Y$. Then $\alpha\leq_1^D\beta$ implies $f(\alpha)=\alpha\leq_1^D f(\beta)$, as we have $\alpha<f(\beta)<\beta$. In the converse direction, $f(\alpha)\leq_1^D f(\beta)=\beta[\eta]$ and $\alpha=f(\alpha)<\eta\leq\beta[\eta]$ yield $\alpha\leq_1^D\eta$. The latter entails $\alpha\leq_1^D\delta$, since $\alpha\not\leq_1^D\delta$ would lead to~$\alpha\not\leq_1^D\alpha'<\eta$. From the assumption $\delta\in C_D(\gamma)$ we also get $\delta\leq_1^D\beta\leq\gamma[\delta]$. It is straightforward to see that~$\leq^D_1$ is transitive. The previous inequalities can thus be combined into $\alpha\leq_1^D\beta$, which was required. For~$\alpha,\beta\in Y$ we show
\begin{equation*}
\alpha\leq_1^D\beta\quad\Leftrightarrow\quad\alpha=\delta\quad\Leftrightarrow\quad\alpha[\eta]=\eta\quad\Leftrightarrow\quad\alpha[\eta]\leq_1^D\beta[\eta].
\end{equation*}
Concerning the middle equivalence, we observe that $D(\delta)=\delta$ and Lemma~\ref{lem:fund-basic}(f) yield $\delta^*=0$ and $\delta[\eta]=D(\eta)=\eta$ (where the last equation follows from~$\eta\in C_D(\gamma)$). This also shows that $\delta<\alpha$ implies $\eta<\alpha[\eta]$, as needed for the converse implication. The first equivalence is similar to the third, so we only provide details for the latter: From $D(\delta)=\delta\leq\alpha<\beta\leq\gamma[\delta]<D(\rho+1)$ we get $\alpha[\eta]<\beta[\eta]\leq\gamma[\delta][\eta]=\gamma[\eta]$, once again by Lemma~\ref{lem:fund-basic}. If we have $\alpha[\eta]=\eta$, we can thus invoke~$\eta\in C_D(\gamma)$ to get $\eta\leq_1^D\gamma[\eta]$ and then $\alpha[\eta]\leq_1^D\beta[\eta]$. To establish the converse implication, we derive a contradiction from the assumption that we have $\eta<\alpha[\eta]\leq_1^D\beta[\eta]$. Write $\alpha[\eta]\simeq(\sigma;\alpha_0,\dots,\alpha_n)_D$ 
with $\alpha_0<\ldots<\alpha_n=\eta$. Given $\eta<\alpha[\eta]<\beta[\eta]$, we can apply Definition~\ref{def:leq_1^D} to $\alpha[\eta]\leq_1^D\beta[\eta]$ and the sets $\{\alpha_0,\dots,\alpha_n\}\subseteq\alpha[\eta]$ and~\mbox{$\{\alpha[\eta]\}\subseteq\beta[\eta]\backslash\alpha[\eta]$}. This yields an $\widetilde\alpha<\alpha[\eta]$ with $\widetilde\alpha\simeq(\sigma;\alpha_0,\dots,\alpha_n)_D$, which contradicts the uniqueness part of Proposition~\ref{prop:representations-exist-unique}. Finally, we establish
\begin{equation*}
\alpha\simeq(\sigma;\alpha_0,\dots,\alpha_{n-1})_D\quad\Rightarrow\quad f(\alpha)\simeq(\sigma;f(\alpha_0),\dots,f(\alpha_{n-1}))_D
\end{equation*}
for arbitrary $\alpha,\alpha_0,\dots,\alpha_{n-1}\in X\cup Y$. Note that the criterion from Theorem~\ref{prop:leq_1^D-Sigma_1-elem}(c) does not require us to verify the converse implication (as the latter is automatic when $X\cup Y$ has suitable closure properties). In case $\{\alpha_0,\dots,\alpha_{n-1}\}\subseteq\delta$ we have $\alpha<D(\delta)=\delta$ (by Proposition~\ref{prop:representations-exist-unique}), so that $f:X\cup Y\to X\cup\widetilde Y$ does not move any of the relevant parameters. Also, $\alpha_{n-1}>\delta$ would entail $D(\delta+1)\leq\alpha\notin X\cup Y$. The only interesting case is thus $\alpha_{n-1}=\delta$ (if~$\delta\in Y$). Here we have $\alpha\in Y$ and
\begin{equation*}
f(\alpha)=\alpha[\eta]\simeq(\sigma;\alpha_0,\dots,\alpha_{n-2},\eta)_D.
\end{equation*}
For $i<n-1$ we have $\alpha_i<\alpha_{n-1}=\delta$, thus $\alpha_i\in X$ and $f(\alpha_i)=\alpha_i$. It remains to show $f(\delta)=\eta$. Due to $\delta,\eta\in C_D(\gamma)$ we have $D(\delta)=\delta$ and $D(\eta)=\eta$. By part~(f) of Lemma~\ref{lem:fund-basic} we now get $f(\delta)=\delta[\eta]=D(\delta)[\eta]=D(\eta)=\eta$, as needed.
\end{proof}

As promised, we can derive the following. Let us recall that a class is called closed if it contains all its limit points.

\begin{corollary}\label{cor:club-closed}
The class $C_D(\gamma)$ is closed for each~$\gamma\geq D(0)$.
\end{corollary}
\begin{proof}
First observe that $\rho\in C_D(\gamma)$ entails $\rho\leq_1^D D(\rho)\leq\gamma[\rho]$ and then $D(\rho)=\rho$, as in the proof of Lemma~\ref{lem:clubs-basis}. Since $\alpha\mapsto D(\alpha)$ is a normal function, we obtain $D(\delta)=\delta$ for any given limit point $\delta$ of~$C_D(\gamma)$. To conclude, we establish $\delta\leq_1^D\eta$ by induction from~$\eta=\delta$ up to $\eta=\gamma[\delta]$. Base case and limit step are immediate by Definition~\ref{def:leq_1^D}. Let us now consider the step from~$\eta$ to $\eta+1\leq\gamma[\delta]<D(\delta+1)$. In view of $D(\delta)=\delta\leq\eta<D(\delta+1)$ we get $\delta\geq\eta^*$ and $\eta[\delta]=\eta$, by Lemma~\ref{lem:fund-basic}. For the induction step we must thus establish $\delta\leq^D_1\eta[\delta]+1$. Due to Proposition~\ref{prop:club-successor}, it suffices to show that $\delta$ is an element and a limit point of~$C_D(\eta)$. In view of~$\eta[\delta]=\eta$ we get $\delta\in C_D(\eta)$ from the induction hypothesis . Now consider an arbitrary~$\alpha<\delta$. We must find a $\beta\in C_D(\eta)$ with~$\alpha<\beta<\delta$. As~$\eta^*$ cannot be a limit (cf.~Definition~\ref{def:fund-seq}), we must have $\eta^*<\delta$. This allows us to pick a $\beta\supseteq\{\alpha,\eta^*\}$ in $C_D(\gamma)\cap\delta$, since~$\delta$ was assumed to be a limit point of this set. Using Lemma~\ref{lem:fund-basic}, we see that $D(\delta)\leq\eta<\gamma[\delta]<D(\delta+1)$ entails $\eta[\beta]<\gamma[\delta][\beta]=\gamma[\beta]$. Since $\beta\in C_D(\gamma)$ provides $\beta\leq_1^D\gamma[\beta]$, we can now conclude $\beta\leq_1^D\eta[\beta]$, as needed for~$\beta\in C_D(\eta)$.
\end{proof}

To formulate the limit step, we fix an ordinal~$\Omega\geq D(0)$. We will later assume that $\Omega$ is the height of an admissible set, but this is not required yet.

\begin{proposition}\label{prop:club-limit}
Let us consider a limit ordinal~$\gamma$ with $D(\Omega)<\gamma<D(\Omega+1)$. For each~$\eta<\Omega$, we put
\begin{equation*}
F_D(\gamma,\eta):=\bigcap\{C_D(\beta)\,|\,\Omega\leq\beta<\gamma\text{ and }\beta^*\leq\eta\}.
\end{equation*}
We then have $\delta\in C_D(\gamma)$ for any limit ordinal~$\delta$ that satisfies $\gamma^*\leq\delta=D(\delta)<\Omega$ as well as $\delta\in F_D(\gamma,\eta)$ for all~$\eta<\delta$.
\end{proposition}

Before we prove the proposition, we sketch how it fits into our inductive argument (see below for details): By induction hypothesis, each of the sets $C_D(\beta)\cap\Omega$ will be club in~$\Omega$. The assumption $\beta^*\leq\eta$ ensures that $F_D(\gamma,\eta)\cap\Omega$ is the intersection of less than~$\Omega$ clubs, and hence club itself. From the proposition we learn that $C_D(\gamma)\cap\Omega$ contains (essentially) the diagonal intersection over the clubs $F_D(\gamma,\eta)\cap\Omega$. Hence $C_D(\gamma)\cap\Omega$ is unbounded in~$\Omega$, and club by Corollary~\ref{cor:club-closed}.

\begin{proof}
From Lemma~\ref{lem:fund-basic} we know that $\gamma[\delta]$ is a limit. Hence the desired conclusion $\delta\leq_1^D\gamma[\delta]$ reduces to $\delta\leq_1^D\alpha$ for $D(\delta)\leq\alpha<\gamma[\delta]$. For any such~$\alpha$ we have $\delta\geq\alpha^*$, which allows us to consider $\beta:=\alpha[\Omega]\geq\Omega$. Using Lemma~\ref{lem:fund-basic}, we compute
\begin{equation*}
\beta[\delta]=\alpha[\Omega][\delta]=\alpha[\delta]=\alpha<\gamma[\delta]
\end{equation*}
and infer $\beta<\gamma$. We also have $\beta^*=\alpha^*<\delta$ (since $\delta$ is a limit), so that we get
\begin{equation*}
\delta\in F_D(\gamma,\beta^*)\subseteq C_D(\beta).
\end{equation*}
In view of $\beta[\delta]=\alpha[\Omega][\delta]=\alpha[\delta]=\alpha$ this yields $\delta\leq_1^D\alpha$, as required.
\end{proof}

As mentioned above, we want to use induction from~$\gamma=\Omega=D(\Omega)$ up to arbitrary~$\gamma<D(\Omega+1)$ to show that $C_D(\gamma)\cap\Omega$ is club in~$\Omega$. If~$\Omega$ was a regular cardinal, this would follow from the previous propositions and standard facts (the limit points of each club form another club, and clubs are closed under diagonal intersections). In the following we recover these facts under the assumption that $\Omega=\mathbb A\cap\ord$ is the height of an admissible set~$\mathbb A\ni D$ (where $D:\nat\to\ord$ is the set-sized object from Definition~\ref{def:pre-dilator}, not its class-sized extension~$D:\ord\to\ord$). For this purpose, we would like to have a $\Delta$-definition of $C_D(\gamma)\cap\Omega$ in~$\mathbb A$, which should be uniform in~$\gamma$. If we take this desideratum literal, then it is impossible to satisfy, simply because we are interested in ordinals~$\gamma\notin\mathbb A$. However, we can get very close: The ordinals~$\gamma\in D(\Omega+1)\backslash D(\Omega)$ are those with representations
\begin{equation*}
\gamma\simeq(\sigma;\gamma_0,\dots,\gamma_{n-1},\Omega)_D
\end{equation*}
that have last entry~$\Omega$. Given that the latter is fixed, we can omit it and write
\begin{equation*}
\gamma\sim\langle\sigma;\gamma_0,\dots,\gamma_{n-1}\rangle_D.
\end{equation*}
This yields a bijection between $D(\Omega+1)\backslash D(\Omega)$ and the collection of expressions that appear on the right, which we denote by
\begin{equation*}
\mathbb D:=\{\langle\sigma;\gamma_0,\dots,\gamma_{n-1}\rangle_D\,|\,(\sigma,n+1)\in\tr(D)\text{ and }\gamma_0<\ldots<\gamma_{n-1}<\Omega\}.
\end{equation*}
We order~$\mathbb D$ so that our bijection $\mathbb D\cong D(\Omega+1)\backslash D(\Omega)$ becomes an isomorphism. Membership in and the order relation on~$\mathbb D$ can be decided by primitive recursive set functions with parameter~$D$ (cf.~Definition~\ref{def:dil-extend} and the discussion in~\cite[Section~2]{freund-computable}). In particular, we obtain a $\Delta$-definition of the order~$\mathbb D\subseteq\mathbb A$ in the admissible set~$\mathbb A$. This allows for an alternative approach to our club sets:

\begin{definition}\label{def:club-admissible}
For each element $\rho=\langle\sigma;\gamma_0,\dots,\gamma_{n-1}\rangle_D\in\mathbb D$, we define
\begin{equation*}
\overline C_D(\rho):=C_D(\gamma)\cap\Omega\quad\text{for }\gamma\sim\langle\sigma;\gamma_0,\dots,\gamma_{n-1}\rangle_D.
\end{equation*}
\end{definition}

As promised, we get the following:

\begin{proposition}\label{prop:clubs-Delta}
If~$D$ is a normal dilator and $\mathbb A\ni D$ an admissible set, then
\begin{equation*}
\{(\delta,\rho)\in\Omega\times\mathbb D\,|\,\delta\in\overline C_D(\rho)\}\subseteq\mathbb A^2
\end{equation*}
is $\Delta$-definable in~$\mathbb A$.
\end{proposition}
\begin{proof}
If $\rho\in\mathbb D$ and $\gamma\in D(\Omega+1)$ are related as in Definition~\ref{def:club-admissible}, we set $\rho^+:=\gamma^*$ and $\rho\langle\delta\rangle:=\gamma[\delta]$ for $\rho^+\leq\delta<\Omega$. Then $\delta\in\overline C_D(\rho)$ is equivalent to the conjunction of $\delta\geq\rho^+$ and $\delta\leq_1^D\rho\langle\delta\rangle$. In order to obtain the claim from the proposition, it suffices to show that the functions $\rho\mapsto\rho^+$, $(\delta,\rho)\mapsto\rho\langle\delta\rangle$ and $\beta\mapsto{\leq_1^D\!\restriction\!(\beta\times\beta)}$ (with the obvious domains of definition) are $\Sigma$-definable and total in~$\mathbb A$. For the first function this is obvious, as $\rho^+=\sup\{\gamma_i+1\,|\,i<n\}$ can be read off from the expression~$\rho=\langle\sigma;\gamma_0,\dots,\gamma_{n-1}\rangle_D\in\mathbb D$. Concerning the second function, we observe that $\rho\langle\delta\rangle$ is characterized by
\begin{equation*}
\eta_{\delta+1}(\rho\langle\delta\rangle)=(\sigma;\gamma_0,\dots,\gamma_{n-1},\delta;\delta+1)_D,
\end{equation*}
for the unique isomorphism $\eta_{\delta+1}:D(\delta+1)\to\overline D(\delta+1)$ (cf.~Definitions~\ref{def:dil-ordinals} and~\ref{def:representation}). Crucially, the function that maps a well-ordered set~$X$ (in the sense of the universe) to its collapse $c:X\to\alpha$ onto an ordinal is $\Sigma$-definable and total in admissible sets (by~\cite[Theorem~4.6]{jaeger-admissibles}, cf.~also~\cite[Exercise~V.6.12]{barwise-admissible}). Together with the fact that~$\alpha\mapsto\overline D(\alpha)$ is primitive recursive, it follows that $\delta\mapsto\eta_\delta$ and hence $(\delta,\rho)\mapsto\rho\langle\delta\rangle$ is $\Sigma$-definable and total in~$\mathbb A$. The same is true for the function $\beta\mapsto{\leq_1^D\!\restriction\!(\beta\times\beta)}$, since its values are primitive recursive in~$\eta_\beta:D(\beta)\to\overline D(\beta)$, by the discussion that follows Definition~\ref{def:leq_1^D} above.
\end{proof}

The following theorem and its corollary are the main results of this section.

\begin{theorem}
Consider a normal dilator~$D$ and an admissible set~$\mathbb A\ni D$ with height $\Omega=\ord\cap\mathbb A$. For each ordinal $\gamma\in D(\Omega+1)\backslash\Omega$, the set $C_D(\gamma)\cap\Omega$ is closed and unbounded in~$\Omega$.
\end{theorem}
\begin{proof}
In the proof of Proposition~\ref{prop:clubs-Delta} we have seen that $\mathbb A$ is closed under the operation~$\alpha\mapsto D(\alpha)$. Due to the continuity of normal functions at limit ordinals, we can conclude $D(\Omega)=\Omega$. In particular, this yields~$\gamma\geq D(0)$ for any~$\gamma\geq\Omega$, which is needed to ensure that $\gamma^*$ and $C_D(\gamma)$ are defined (cf.~Definition~\ref{def:fund-seq}). From Corollary~\ref{cor:club-closed} we know that the set $C_D(\gamma)\cap\Omega$ is closed in~$\Omega$. To show that it is unbounded, we argue by induction from~$\gamma=\Omega=D(\Omega)$ up to arbitrary~\mbox{$\gamma<D(\Omega+1)$}. Let us point out that the induction statement is primitive recursive with the function $\eta_{\Omega+1}:D(\Omega+1)\to\overline D(\Omega+1)$ as parameter, by the proof of Proposition~\ref{prop:clubs-Delta}. In the base case of~$\gamma=\Omega=D(\Omega)$, we invoke Lemma~\ref{lem:clubs-basis} to get
\begin{equation*}
C_D(\Omega)\cap\Omega=\{\delta<\Omega\,|\,D(\delta)=\delta\}.
\end{equation*}
The usual argument that normal functions have arbitrarily large fixed points can be accommodated in our setting: Given an arbitrary~$\alpha_0<\Omega$, we set $\alpha_{n+1}:=D(\alpha_n)$ to get a function $\mathbb N\ni n\mapsto\alpha_n<\Omega$ that is $\Sigma$-definable in~$\mathbb A$. We take $D\in\mathbb A$ to entail $\mathbb N\in\mathbb A$ (as~$D$ is a functor with domain~$\nat$). By~$\Sigma$-collection in~$\mathbb A$ (cf.~\cite[Section~I.4.4]{barwise-admissible}) we then obtain $\alpha_\infty:=\sup\{\alpha_n\,|\,n\in\mathbb N\}<\Omega$. Unless we already have a fixed point $D(\alpha_n)=\alpha_n$ for some~$n\in\mathbb N$, the ordinal~$\alpha_\infty$ is a limit, so that we get
\begin{equation*}
D(\alpha_\infty)=\sup\{D(\alpha_n)\,|\,n\in\mathbb N\}=\sup\{\alpha_{n+1}\,|\,n\in\mathbb N\}=\alpha_\infty
\end{equation*}
by continuity of normal functions. For the successor step of our induction, we show
\begin{equation*}
\{\delta<\Omega\,|\,\text{$\delta\geq(\gamma+1)^*$ is a limit point of $C_D(\gamma)$}\}\subseteq C_D(\gamma+1)\cap\Omega.
\end{equation*}
Given any element~$\delta$ of the left side, Corollary~\ref{cor:club-closed} yields $\delta\in C_D(\gamma)$. We can apply Proposition~\ref{prop:club-successor} to get~$\delta\leq_1^D\gamma[\delta]+1$. By Lemma~\ref{lem:fund-basic} we have $\gamma[\delta]+1=(\gamma+1)[\delta]$, so that we obtain $\delta\in C_D(\gamma+1)$, as desired. Let us note that $\gamma+1<D(\Omega+1)$ entails $(\gamma+1)^*<\Omega$ (since~$\Omega$ is a limit). To complete the successor step of our induction, we construct arbitrarily large limit points of~$C_D(\gamma)\cap\Omega$, which is club by induction hypothesis. Crucially, Proposition~\ref{prop:clubs-Delta} ensures that $C_D(\gamma)\cap\Omega$ (which is equal to $\overline C_D(\rho)$ for the appropriate~$\rho\in\mathbb D$) is $\Delta$-definable in~$\mathbb A$. Once this is known, we can rely on the usual construction: Given an arbitrary start value~$\alpha_0<\Omega$, we inductively choose $\alpha_{n+1}>\alpha_n$ minimal with $\alpha_{n+1}\in C_D(\gamma)\cap\Omega$. Since this last set is $\Delta$-definable in~$\mathbb A$, we obtain a $\Sigma$-definable function $\mathbb N\ni\alpha\mapsto\alpha_n<\Omega$. As above, we get $\alpha_\infty:=\sup\{\alpha_n\,|\,n\in\mathbb N\}<\Omega$ by $\Sigma$-collection in~$\mathbb A$. The construction ensures that $\alpha_\infty$ is a limit point of~$C_D(\gamma)$. It remains to consider the case of a limit ordinal~$\gamma\in D(\Omega+1)\backslash(\Omega+1)$. By Proposition~\ref{prop:club-limit}, the set $C_D(\gamma)$ contains the intersection of the $\Delta$-definable club $\{\delta<\Omega\,|\,\delta=D(\delta)\geq\gamma^*\text{ is limit}\}$ with the set
\begin{equation}\label{eq:diag-intersect}
\{\delta<\Omega\,|\,\delta\in F_D(\gamma,\eta)\text{ for all }\eta<\delta\}.
\end{equation}
It suffices to show that the latter is a $\Delta$-definable club as well (since the intersection of two such clubs is easily seen to be club itself). Assume that $\gamma$ and $\rho\in\mathbb D$ are related as in Definition~\ref{def:club-admissible}, and let $\mathbb D\ni\pi\mapsto\pi^+\in\Omega$ be the function from the proof of Proposition~\ref{prop:clubs-Delta}. We then have
\begin{equation*}
F_D(\gamma,\eta)\cap\Omega=\bigcap\{\overline C_D(\pi)\,|\,\pi<_{\mathbb D}\rho\text{ and }\pi^+\leq\eta\}.
\end{equation*}
We write $\mathbb D(\eta):=\{\pi\in\mathbb D\,|\,\pi^+\leq\eta\}$ and observe
\begin{equation*}
\mathbb D(\eta)=\{\langle\sigma;\gamma_0,\dots,\gamma_{n-1}\rangle_D\,|\,(\sigma,n+1)\in\tr(D)\text{ and }\gamma_0<\ldots<\gamma_{n-1}<\eta\}.
\end{equation*}
Hence $\eta\mapsto\mathbb D(\eta)$ is a primitive recursive set function (with parameter~$D$), and in particular~$\Sigma$-definable and total in~$\mathbb A$. As $\delta\in F_D(\gamma,\eta)$ is equivalent to
\begin{equation*}
\forall \pi\in\mathbb D(\eta)\,[\pi<_{\mathbb D}\rho\to\delta\in\overline C_D(\pi)],
\end{equation*}
we can conclude that $\{(\delta,\eta)\in\Omega^2\,|\,\delta\in F_D(\gamma,\eta)\}\subseteq\mathbb A^2$ and the collection in~(\ref{eq:diag-intersect}) are $\Delta$-definable. Let us now show that each of the sets~$F_D(\gamma,\eta)\cap\Omega$ is club, by adapting the usual argument to our setting: From the induction hypothesis we know that $\overline C_D(\pi)$ is club for all $\pi<_{\mathbb D}\rho$. Starting with an arbitrary~$\alpha_0<\Omega$, we construct a $\Sigma$-definable function $\mathbb N\ni n\mapsto\alpha_n<\Omega$ as follows: In the step, consider the function
\begin{gather*}
f_n:\{\pi\in\mathbb D(\eta)\,|\,\pi<_{\mathbb D}\rho\}\to\Omega,\\
f_n(\pi):=\min\{\alpha>\alpha_n\,|\,\alpha\in\overline C_D(\pi)\}.
\end{gather*}
By $\Delta$-separation and $\Sigma$-replacement (see~\cite[Theorems~4.5 and~4.6]{barwise-admissible}) we get $f_n\in\mathbb A$. This allows us to set
\begin{equation*}
\alpha_{n+1}:=\sup\{f_n(\pi)\,|\,\pi\in\mathbb D(\eta)\text{ and }\pi<_{\mathbb D}\rho\}<\Omega.
\end{equation*}
Another application of $\Sigma$-collection yields $\alpha_\infty:=\sup\{\alpha_n\,|\,n\in\mathbb N\}<\Omega$. To see that $\alpha_\infty$ is a limit point (and hence an element) of each set~$\overline C_D(\pi)$, we consider an arbitrary ordinal $\alpha<\alpha_\infty$. We then have $\alpha\leq\alpha_n$ for some~$n\in\mathbb N$, so that we indeed get $\alpha<f_n(\pi)\in\overline C_D(\pi)\cap\alpha_{\infty}$. Finally, we deduce that the collection in~(\ref{eq:diag-intersect}) is club. As mentioned before, this collection is a diagonal intersection. The usual argument goes through in our setting: Start with an arbitrary~$\alpha_0<\Omega$. In the step, set
\begin{equation*}
\alpha_{n+1}:=\min\{\alpha<\Omega\,|\,\alpha>\alpha_n\text{ and }\alpha\in F_D(\gamma,\alpha_n)\}.
\end{equation*}
This is possible because $F_D(\gamma,\eta)\cap\Omega$ is club for~$\eta<\Omega$, as shown above. Since the relation $\alpha\in F_D(\gamma,\eta)$ is $\Delta$-definable, the resulting function $n\mapsto\alpha_n$ is \mbox{$\Sigma$-definable} in~$\mathbb A$. Once again we obtain $\alpha_\infty:=\sup\{\alpha_n\,|\,n\in\mathbb N\}<\Omega$. In order to see that $\alpha_\infty$ is contained in the diagonal intersection~(\ref{eq:diag-intersect}), we consider an arbitrary~$\eta<\alpha_\infty$. Pick an~$n\in\mathbb N$ with $\eta\leq\alpha_n$. For any~$N>n$ we get $\alpha_N\in F_D(\gamma,\alpha_{N-1})\subseteq F_D(\gamma,\eta)$, where the inclusion is an easy consequence of~$\eta\leq\alpha_{N-1}$. This shows that $\alpha_\infty$ is a limit point and hence an element of~$F_D(\gamma,\eta)$, as needed.
\end{proof}

It is straightforward to derive the following:

\begin{corollary}\label{cor:Omega-elementary}
Let~$D$ be a normal dilator. If $\Omega=\ord\cap\mathbb A$ is the ordinal height of an admissible set $\mathbb A\ni D$, then we have $\Omega\leq_1^D D(\Omega+1)$.
\end{corollary}
\begin{proof}
As in the previous proof we get $\Omega=D(\Omega)$. It suffices to show that we have $\Omega\leq_1^D\gamma+1$ whenever $D(\Omega)\leq\gamma<D(\Omega+1)$, no matter if $D(\Omega+1)$ is a limit or not. From the previous theorem we learn that $\Omega$ is a limit point of~$C_D(\gamma)$. By Corollary~\ref{cor:club-closed} this implies~$\Omega\in C_D(\gamma)$, so that Proposition~\ref{prop:club-successor} yields~$\Omega\leq_1^D\gamma[\Omega]+1$. To conclude, we note that $\gamma[\Omega]=\gamma$ holds by Lemma~\ref{lem:fund-basic}.
\end{proof}

Finally, we combine our results to establish the theorem from the introduction:

\begin{proof}[Proof of~Theorem~\ref{thm:pattern-admissibles}]
To show that~(i) implies~(ii), we make use of Theorem~\ref{thm:from-freund-equivalence} (originally proved in~\cite{freund-thesis,freund-equivalence,freund-categorical,freund-computable}). Aiming at statement~(i) of that theorem, we consider an arbitrary dilator~$D$. Form the normal dilator~$\Sigma D$ from Definition~\ref{def:Sigma-D}. By statement~(i) of  Theorem~\ref{thm:pattern-admissibles} we get an ordinal~$\Omega$ with~$\Omega\leq_1^{\Sigma D}\Sigma D(\Omega+1)$. From Theorem~\ref{thm:patterns-to-collapse} we learn that $\Omega$ is a Bachmann-Howard fixed point of~$D$, as needed to satisfy statement~(i) of Theorem~\ref{thm:from-freund-equivalence}. To establish the implication from~(ii) to~(i) in Theorem~\ref{thm:pattern-admissibles}, we consider a normal dilator~$D$. Invoking~(ii), we pick an admissible set~$\mathbb A\ni D$ (where $D$ denotes the set-sized object from Definition~\ref{def:pre-dilator}, rather than its class-sized extension due to Definition~\ref{def:dil-ordinals}). By Corollary~\ref{cor:Omega-elementary} we have $\Omega\leq_1^D D(\Omega+1)$ for~$\Omega=\ord\cap\mathbb A$, as required by statement~(i) of Theorem~\ref{thm:pattern-admissibles}.
\end{proof}

\bibliographystyle{amsplain}
\bibliography{Patterns_Bachmann-Howard}

\end{document}